\documentclass[preprint,10pt]{elsarticle}
\usepackage{amssymb}
\usepackage{amsthm}
\usepackage{graphics}
\usepackage{amsmath}
\usepackage[dvipsnames,usenames]{color}
\usepackage{subfigure}
\usepackage{bm}
\newtheorem{thm}{Theorem}[section]

\newtheorem{lem}[thm]{Lemma}

\newtheorem{ass}[thm]{Assumption}
\newtheorem{rem}{Remark}

\theoremstyle{definition}

\numberwithin{equation}{section}

\journal{.}

\begin{document}

\begin{frontmatter}



\title{McKean-Vlasov stochastic differential equations with super-linear measure arguments: well-posedness and propagation of chaos}

\author[label1]{Zhuoqi Liu}
\author[label1]{Qian Guo}
\author[label2]{Shuaibin Gao}
\author[label3]{Chenggui Yuan}

\address[label1]{Department of Mathematics, Shanghai Normal University, Shanghai 200234, China}
\address[label2]{School of Mathematics and Statistics, South-Central Minzu University,
Wuhan 430074, China}
\address[label3]{Department of Mathematics, Swansea University, Bay Campus, Swansea SA1 8EN, UK}

\begin{abstract}
This paper studies McKean-Vlasov stochastic differential equations (MVSDEs) whose drift coefficients grow super-linearly in both state variables and measure arguments, and whose diffusion coefficients exhibit super-linear growth in the state variables. 
By constructing an Euler-like sequence, we establish the strong well-posedness of such MVSDEs under a locally monotone condition.	Furthermore, the propagation of chaos is studied on both finite and infinite horizons, demonstrating convergence of the interacting particle system to the corresponding non-interacting system.
To illustrate the rationality of the theoretical results, we provide examples whose drifts contain the high powers and multiple integrals of distributions, with numerical simulations presented in Section 6.
\end{abstract}


\end{frontmatter}

% \linenumbers


\section{Introduction}\label{secc1}
The theories of MVSDEs have attracted considerable attention due to their wide range of applications in fields such as physics, biology, and social science  \cite{appli1,appli2,appli3}. The coefficients of MVSDEs not only depend on the state variables but also relate to the law information of state variables, so they are also called distribution-dependent SDEs or mean-field SDEs. 
Let $b(\cdot,\cdot)$ and $\sigma(\cdot,\cdot)$ be continuous on the space $\mathbb{R}^d\times\mathcal{P}_{2}(\mathbb{R}^d)$ with $b:\mathbb{R}^d\times\mathcal{P}_{2}(\mathbb{R}^d)\rightarrow\mathbb{R}^d$ and $\sigma:\mathbb{R}^d\times\mathcal{P}_{2}(\mathbb{R}^d)\rightarrow\mathbb{R}^{d\times m_1}$. Let $(W_t, t\geq 0)$ be a standard $m_1$-dimensional Brownian motion defined on a complete filtered
probability space $\big(\Omega,\mathcal{F},(\mathcal{F}_t)_{t\geq0},\mathbb{P}\big)$.
This paper is devoted to investigating the MVSDEs of the form:
\begin{equation}\label{MV}
	dX_{t}=b(X_{t},\mu_{t})dt+\sigma(X_{t},\mu_{t})dW_{t},~~~X_0=\xi,
\end{equation}
where $\mu_t:=\mathcal{L}(X_t)$ denotes the law of $X_t$. We use the notation $\mathcal{L}(\cdot)$ to represent the law of a random variable throughout the whole paper.
The initial value $\xi$ of (\ref{MV}) is a random variable which belongs to $\mathbb{L}^p(\mathbb{R}^d)$, which is the set of random variables $\zeta$ satisfying $\mathbb{E}|\zeta|^p<\infty$ for $p\geq1$.
The study of MVSDEs originated in \cite{Mck}, which was inspired by the kinetic theory in \cite{Kac}. 
The strong well-posedness of solutions to MVSDEs has been extensively investigated under different conditions with respect to (w.r.t.) the state variables in coefficients such as linear growth and global Lipschitz condition \cite{topic}, one-sided Lipschitz drift and global Lipschitz diffusion \cite{wangfy}, and super-linear drifts and diffusions \cite{kumar}.
As for the MVSDEs with irregular coefficients, the well-posedness theories were also established in \cite{baoirr,holder,tamedhuagui}.
The common feature of the equations in \cite{topic,wangfy,kumar,baoirr,holder,tamedhuagui,liyun} is that the distribution components in the coefficients satisfy the global Lipschitz conditions in the Wasserstein distance sense.

However, consider the following two  typical classes of MVSDEs with super-linear measure arguments:
\begin{equation}\label{1}
	dX_{t}=\Big(-X_{t}-X_{t}\Big(\int_{\mathbb{R}^d}x\mu_t(dx)\Big)^2\Big)dt+\Big(X_{t}+\int_{\mathbb{R}^d}x\mu_t(dx)\Big)dW_{t},
\end{equation}
and
\begin{equation}\label{2}
	dX_{t}=\iint_{\mathbb{R}^d\times\mathbb{R}^d}b(X_{t},y,z)\mu_{t}(dy)\mu_{t}(dz)dt+\int_{\mathbb{R}^d}\sigma(X_{t},y)\mu_{t}(dy)dW_t,
\end{equation}
where $\mu_{\cdot}$ is the law of $X_{\cdot}$, $b$ is continuously differentiable w.r.t. to $y$ and $z$, $\partial_{y}b$ and $\partial_{z}b$ grow at most linearly in $(y,z)$. The drift coefficients w.r.t. measure arguments in (\ref{1}) and (\ref{2}) no longer satisfy the global Lipschitz condition, but satisfy the local Lipschitz condition. 
To prove this viewpoint rigorously, some notations are given first.
Let $|\cdot|$ and $\|\cdot\|$ denote the Euclidean norm in $\mathbb{R}^d$ and the trace norm in $\mathbb{R}^{d\times m_1}$, respectively. 
Denote $a\vee b=\max\{a,b\}$ and $a\wedge b=\min\{a,b\}$ for any real numbers $a,b$.
Let $\mathcal{P}(\mathbb{R}^d)$ be the space of all probability measures on $\big(\mathbb{R}^d,\mathcal{B}(\mathbb{R}^d)\big)$, where $\mathcal{B}(\mathbb{R}^d)$ stands for the Borel $\sigma$-field over $\mathbb{R}^d$. 
For $p\geq1$, define 
\begin{equation*}
	\mathcal{P}_{p}(\mathbb{R}^d)=\left\lbrace \mu\in\mathcal{P}(\mathbb{R}^d):\|\mu\|:=\int_{\mathbb{R}^d}|x|^p\mu(dx)<\infty\right\rbrace ,
\end{equation*}
which is the subset of probability measures with bounded moments. 
The Wasserstein distance between $\mu,\nu\in\mathcal{P}_{p}(\mathbb{R}^d)$ is defined by
\begin{equation*}
	\mathbb{W}_p(\mu,\nu):=\inf_{\pi\in\mathcal{C}(\mu,\nu)}\left(  \int_{\mathbb{R}^d\times\mathbb{R}^d}|x-y|^p\pi(dx,dy)\right)^\frac{1}{p} ,
\end{equation*}
where  $\mathcal{C}(\mu,\nu)$ is the family of all couplings for $\mu,\nu$, i.e., $\pi(\cdot, \mathbb{R}^d)= \mu(\cdot)$ and $\pi( \mathbb{R}^d,\cdot)= \nu(\cdot)$. 

First, we show that the drift in (\ref{1}) satisfies the local Lipschitz condition w.r.t. measure.
Let us introduce a function $h:\mathcal{P}_2(\mathbb{R}^d)\rightarrow\mathbb{R}$, which is defined by $h(\mu)=\big(\int_{\mathbb{R}^d}x\mu(dx)\big)^2$, $\mu\in\mathcal{P}_2(\mathbb{R}^d)$. Thus, for measures $\mu,\nu\in\mathcal{P}_2(\mathbb{R}^d)$, it holds that
\begin{equation}
	\begin{split}
		h(\mu)-h(\nu)&=\big(\int_{\mathbb{R}^d}x\mu(dx)\big)^2-\big(\int_{\mathbb{R}^d}y\nu(dy)\big)^2\\
		&=\big(\int_{\mathbb{R}^d}x\mu(dx)+\int_{\mathbb{R}^d}y\nu(dy)\big)\big(\int_{\mathbb{R}^d}x\mu(dx)-\int_{\mathbb{R}^d}y\nu(dy)\big)	\\
		&\leq C\big(1+\|\mu\|_2+\|\nu\|_2\big)\mathbb{W}_1(\mu,\nu).
	\end{split}
\end{equation}

Next, we reveal that the drift in (\ref{2}) also satisfies the local Lipschitz condition w.r.t measure argument. 
%	The required definitions and properties of Lions derivatives are stated in the Appendix. 
%	For simplicity, we use notation $\mathcal{L}(\cdot)$ to represent the law of random variable throughout the whole paper.
Let $u$: $\mathcal{P}_2(\mathbb{R}^d)\rightarrow\mathbb{R}$ be a function and its lifted function $\tilde{u}$ is defined by $\tilde{u}(X)=u(\mu )$, $X\in \mathbb{L}^2(\mathbb{R}^d)$ and $\mu=\mathcal{L}(X)$. 
The function $u$ is said to be L-differentiable at $\mu_0\in\mathcal{P}_2(\mathbb{R}^d)$ if there exists a random variable $X_0$ with $\mathcal{L}(X_0)=\mu_0$, such that the lifted function $\tilde{u}$ is Fr\'{e}chet differentiable at $X_0$, and its derivative is denoted by $D\tilde{u}$. From Riez's representation theorem,  $D\tilde{u}$ can be identified as a $\sigma(X_0)$-measurable random variable in $\mathbb{L}^2(\mathbb{R}^d)$, and is used to represent the function $\partial_{\mu} u(\mu_0)(\cdot):\mathbb{R}^d\rightarrow\mathbb{R}^d$, which is regarded as the L-derivative of $u$ at $\mu_0$.
According to  \cite[Remark 2.3]{poly2}, it holds that
\begin{equation}\label{uv}
	\begin{split}
		\big|u({\mu})-u(\nu)\big|=&\left|\int_{0}^{1}\mathbb{E}\big
		[\partial_{\mu}u\big(\mathcal{L}(\lambda X+(1-\lambda) Y)\big)\big(\lambda X+(1-\lambda) Y\big)\big(X-Y\big)\big]d\lambda\right|\\
		\leq&\|D\tilde{u}\|_{2}\mathbb{W}_{2}(\mu,\nu),\\
	\end{split}
\end{equation}
where $X$ and $Y$ are two independent random variables in $\mathbb{L}^2(\mathbb{R}^d)$ with $\mu:=\mathcal{L}(X)$ and $\nu:=\mathcal{L}(Y)$, and
\begin{equation*}
	\|D\tilde{u}\|_{2}:=	\Big(\mathbb{E}\big<\partial_{\mu}u(\mathcal{L}({\cdot}))(\cdot),\partial_{\mu}u(\mathcal{L}({\cdot}))(\cdot)\big>\Big)^\frac{1}{2}.
\end{equation*}
The key point to estimate the right-hand side of (\ref{uv})  is to evaluate $\|D\tilde{u}\|_{2}$. 
Let
\begin{equation}\label{fff}
	u({\mu})=\iint_{\mathbb{R}^d\times\mathbb{R}^d}f(x,y)\mu_{s}(dx)\mu_{s}(dy),
\end{equation}
which is a quadratic function of measure $\mu$ with $f:\mathbb{R}^d\times \mathbb{R}^d\rightarrow\mathbb{R}$. Assume $f$ is continuously differentiable in $(x,y)$ alongside the first partial derivative $\partial_{x}$ and $\partial_{y}$ are mostly linear growth in $(x,y)$, that is to say,
\begin{equation}\label{linear}
	\big|\partial_{x}f(x,y)\big|\vee\big|\partial_{y}f(x,y)\big|\leq C(1+|x|+|y|).
\end{equation}
This,  together  with \cite[Example 4, p385]{2018},
one can see that
\begin{equation}\label{111}
	\begin{split}
		\big|\partial_{\mu}u(\mu)(x)\big|^2\leq&\int_{\mathbb{R}^d} \big|\partial_{x}f(x,y)\big|^2\mu(dy)+\int_{\mathbb{R}^d} \big|\partial_{y}f(x,y)\big|^2\mu(dy)\\
		\leq&C\int_{\mathbb{R}^d} \big(1+|x|^{2}+|y|^{2}\big)\mu(dy)\\
		\leq&C \big(1+|x|^{2}+\mathbb{E}|X|^{2}\big).
	\end{split}
\end{equation}
%	where $\mu:=\mathcal{L}(X)$.
Then, by Lagrange's mean value Theorem, it can be deduced that there exists $\lambda^{*}\in(0,1)$ such that 
\begin{equation}\label{Z}
	\|D\tilde{u}\|_{2}=	\Big(\mathbb{E}\big<\partial_{\mu}u(\mathcal{L}_{Z})(Z),\partial_{\mu}u(\mathcal{L}_{Z})(Z)\big>\Big)^\frac{1}{2}
\end{equation}
with $Z=\lambda^{*} X+(1-\lambda^{*}) Y$. In conclusion, associating (\ref{111}) with (\ref{Z}) gives that
\begin{equation}
	\begin{split}
		\Big(\mathbb{E}\big|\partial_{\mu}u(\mathcal{L}_{Z})(Z)\big|^2\Big)^\frac{1}{2}\leq&C \big(1+\mathbb{E}|Z|^2\big)^\frac{1}{2}
		\leq C \big(1+\mathbb{E}|X|^2+\mathbb{E}|Y|^2\big)^\frac{1}{2}.
	\end{split}
\end{equation}
Thus, the measure in (\ref{fff}) satisfies the local Lipschitz condition, that is
\begin{equation}\label{101}
	|u({\mu})-u(\nu)|^2\leq C(1+\|\mu\|_2^2+\|\nu\|^2_2)\mathbb{W}_{2}^2(\mu,\nu).
\end{equation}

To sum up, the drifts of (\ref{1}) and (\ref{2}) both fulfill the above local Lipschitz condition (\ref{101}) but no longer satisfy the global Lipschitz condition in the Wasserstein distance sense.

In fact, a rigorous analysis has been conducted for  MVSDEs with measures in coefficients satisfying the local condition in \cite{first}. When the drift and diffusion coefficients satisfy local Lipschitz conditions w.r.t. state variables and measure variables, the strong well-posedness of McKean-Vlasov equations with jumps was provided in \cite{jump}, as well as the propagation of chaos. However, the uniform boundedness of coefficients was required. The existence of weak solution and the pathwise uniqueness of MVSDE with unbounded coefficient on a domain under measure-dependent Lyapunov conditions were studied in \cite{luk}. 
In \cite{hongliu}, the authors established the strong and weak well-posedness of the MVSDE with a locally monotone condition on the state variables and measure variables, while additionally requiring that the diffusion term be globally Lipschitz in these variables.
For the MVSDEs with drifts and diffusions super-linearly growing in measure and space, in \cite{dossup1}, the diffusion coefficient is assumed to satisfy a super-linear growth condition w.r.t. the measure argument, while in \cite{dossup2}, the growth in the measure argument is restricted to be at most linear. However,
In both of the aforementioned studies, the features of super-linear growth in measure argument are  imposed some conditions on the kernel functions within a designated convolution framework.
Both works rigorously establish the well-posedness of MVSDEs and provide detailed convergence analysis for the associated split-step Euler schemes. Moreover, \cite{dossup1} further proves the exponential ergodicity together with the existence of an invariant distribution for the underlying equation.
Beyond these developments, the integration by parts formulae for MVSDEs were developed in \cite{inter}.

We proceed to highlight the work of Kac in \cite{Kac}, who provided the first rigorous mathematical definition of chaos and introduced the idea that, in time-evolving systems, chaos should be propagated over time, a property hence known as the ``propagation of chaos".
Building upon this foundation, \cite{topic} carried out a more systematic and comprehensive analysis of the propagation of chaos, extending Kac's original idea to a more general setting.
Specifically,
for $N$-interacting particle system of SDEs, when the initial distribution of particles is chaotic, the propagation of chaos theory in \cite{topic} shows that the particles behave increasingly independently as the number of particles $N$ tends to infinity, which means that the empirical measure of the particle system converges to the distribution appearing in the coefficients of the MVSDEs as $N\rightarrow\infty$. 
For MVSDEs whose coefficients satisfy global Lipschitz conditions w.r.t. measure arguments, the propagation of chaos over a finite time horizon has been investigated in \cite{kumar,baoirr,tamedhuagui,liyun,singular}, while the results on uniform-in-time propagation of chaos have been derived in \cite{uni1,uni2}. In contrast, under local Lipschitz conditions w.r.t. measure arguments, the finite-time propagation of chaos has been established in \cite{jump,dossup1,dossup2,spde}.

In comparison with the aforementioned literature, the assumptions on the coefficients considered in this paper are significantly weaker.
The main contributions of this paper are summarized as follows:
\begin{itemize}
	\item The strong well-posedness of the solution to (\ref{MV}) is established under the condition that the drift and diffusion grow super-linearly w.r.t. the state variable, while the measure argument in drift is super-linear in the Wasserstein distance sense.
	\item The propagation of chaos theory in finite horizon is proposed to reveal that the interacting particle system can be used to approximate the non-interacting particle system.
	\item The propagation of chaos theory in infinite horizon is also achieved under the dissipative conditions, which means that the interacting particle system can converge to the associated non-interacting one as time tends to infinity.
\end{itemize}

The remainder of this paper is organized as follows. In Section \ref{secc2}, the main results are presented one by one. In Section \ref{secc3}, the strong well-posedness of (\ref{MV}) is addressed by constructing an Euler-like sequence, which is stated as Theorem \ref{thm1}.  Section \ref{secc4} is
devoted to finishing the proof of Theorem \ref{keythm}.
Section \ref{secc5} aims to complete the proof of Theorem \ref{longpoc}. 
In Section 6, numerical simulations are conducted to verify the theories about the propagation of chaos in finite and infinite horizons.
In the Appendix, we provide the proof of exponential integrability of the solution and the verification of assumptions.

\section{Main results} \label{secc2}

This section presents the main results of this paper, including strong well-posedness and propagation of chaos in finite and infinite horizons. Inspired by the conditions of coefficients in the two examples from Section \ref{secc1}, we impose the following assumptions on  (\ref{MV}).

\begin{ass}\label{ass1}
	For any $R>0$, there exist constants $q,\gamma\geq2$ and $L_{1}>0$ such that for any $x,\bar{x}\in\mathbb{R}^d$ with $|x|\vee|\bar{x}|\leq R$ and $\mu$, $\nu\in\mathcal{P}_{\gamma}(\mathbb{R}^d)$,
	\begin{equation*} 
		\begin{split}
			&2	\langle x-\bar{x},b(x,\mu)-b(\bar{x},\nu)\rangle+(q-1)\|\sigma(x,\mu)-\sigma(\bar{x},\nu)	\|^2\\
			\leq& \big(L(R)+L_{1}\|\mu\|_\gamma^\gamma+L_{1}\|\nu\|_\gamma^\gamma\big)\big(|x-\bar{x}|^2+\mathbb{W}_{2}^2(\mu,\nu)\big),
		\end{split}
	\end{equation*} 
	where $L:[0,\infty)\rightarrow[0,\infty)$ is a function satisfying $L(R)<\infty$.
	%for any $x$, $\bar {x}\in\mathbb{R}^d$ and $\mu\in\mathcal{P}_2(\mathbb{R}^d)$.
\end{ass}

\begin{ass}\label{ass4}
	There exist constants $L_{2}>0$ and $l_{1},l_{2}\geq 1$ such that
	\begin{equation*} 
		|b(x,\mu)-b(\bar{x},\nu)|\leq L_{2}(1+|x|^{l_{1}}+|\bar{x}|^{l_{1}}+\|\mu\|_{l_{1}}^{l_{1}}+\|\nu\|_{l_{1}}^{l_{1}})\left(|x-\bar{x}|+\mathbb{W}_2(\mu,\nu)\right),
	\end{equation*} 
	\begin{equation*} 
		\|\sigma(x,\mu)-\sigma(\bar{x},\mu)	\|\leq L_{2}(1+|x|^{l_{2}}+|\bar{x}|^{l_{2}})|x-\bar{x}|,
	\end{equation*} 
	\begin{equation*} 
		\|\sigma(x,\mu)-\sigma(x,\nu)	\|\leq L_{2}\mathbb{W}_2(\mu,\nu),
	\end{equation*} 
	for any $x,\bar{x}\in\mathbb{R}^d$ and $\mu$, $\nu\in\mathcal{P}_{l_{1}\vee 2}(\mathbb{R}^d)$.
\end{ass}
\begin{ass}\label{ass3}
	There exist constants $p\geq2$ and $L_3>0$ such that
	\begin{equation*} 
		2\langle x,b(x,\mu)
		\rangle+(p-1)\|\sigma(x,\mu)\|^2\leq L_{3}\big(1+|x|^2+\|\mu\|_2^2\big),
	\end{equation*} 
	for any $x\in\mathbb{R}^d$ and $\mu\in\mathcal{P}_2(\mathbb{R}^d)$.
\end{ass}

\begin{ass}\label{ass2}
	There exist constants $L_{4}>0$ and $l_{3}, l_{4}\geq1$ such that
	\begin{equation*} 
		|b(x,\mu)|\leq L_{4}(1+|x|^{l_{3}}+\|\mu\|_{l_{3}}^{l_{3}}),
	\end{equation*} 
	\begin{equation*} 
		\|\sigma(x,\mu)\|\leq L_{4}(1+|x|^{l_{4}}+\|\mu\|_2),
	\end{equation*} 
	for any $x\in\mathbb{R}^d$ and $\mu\in\mathcal{P}_{l_{3}\vee 2}(\mathbb{R}^d)$.
\end{ass}

Obviously, Assumption \ref{ass2} can be derived from Assumption \ref{ass4}, but we still state Assumption \ref{ass2} explicitly to simplify the notations.

\begin{ass}\label{ass6}
	For any $x\in\mathbb{R}^d$, $\mu\in\mathcal{P}_2(\mathbb{R}^d)$, and initial value $\xi\in\mathcal{F}_0$, there exist non-decreasing, measurable, and polynomial functions $f$, $\bar{f}$: $[0,\infty)\rightarrow[0,\infty)$ such that $\mathbb{E}[e^{f(|\xi|)}]<\infty$, and
	\begin{equation*} 
		\begin{split}
			&	\langle \nabla f(|x|),b(x,\mu)
			\rangle+\frac{1}{2}|\nabla f(|x|)\sigma(x,\mu)|^2+\frac{1}{2}\text{trace}\left({\sigma}^{T}(x,\mu)\nabla^2 f(|x|)\sigma(x,\mu)\right)\\\leq& \alpha f(|x|)+\beta \bar{f}\left(\|\mu\|_2^2\right),
		\end{split}
	\end{equation*} 
	and
	\begin{equation}\label{R}
		\lim_{R\rightarrow\infty}\big(f(R)-\kappa L(R)\big)=\infty,
	\end{equation}
	where $\alpha$, $\beta$, and $\kappa$ are some positive constants.
	Here, for $i,j=1,\cdots,d$,
	$$  \nabla f(|x|) = \left( \frac{\partial}{\partial x_1} f(|x|), \dots, \frac{\partial}{\partial x_d} f(|x|) \right),~
	\left[ \nabla^2 f(|x|) \right]_{ij} = \frac{\partial^2}{\partial x_{i} \partial x_{j} }f(|x|).
	$$

\end{ass}
\begin{rem}\label{rree01}
	The Assumption \ref{ass6} is employed to establish the exponential integrability of the solution to (\ref{MV}), which plays a crucial role in the analysis of the existence and uniqueness of the solution to (\ref{MV}) in Theorem \ref{thm1}. A detailed proof of the exponential integrability is provided in the Appendix. For further details on the exponential integrability property, we refer the reader to \cite{xiaojie}.
\end{rem}

\begin{rem}
	If the function $L(R)\leq C\log R$ for $ C>0$ in Assumption \ref{ass1}, the Assumption \ref{ass6} becomes redundant. However, when the function $L(R)$ grows faster than $\log R$, for example, $L(R)=R^a$ for  $a>0$, the Assumption \ref{ass6} is required to guarantee the exponential integrability of the solution.		
\end{rem}

\begin{rem}
	Compared with  \cite{hongliu}, the assumptions in this paper are more relaxed. First, the exponent of the state variable in drift $l_{3}$ is no longer constrained by $\gamma$ in Assumption \ref{ass1}, which means that
	the degree of super-linearity of the state variable in drift is less restrictive. Second, the state variable in diffusion is allowed to be super-linear.
	Third, the assumption of the exponential integrability in \cite{hongliu}  is difficult to verify directly, whereas our Assumption \ref{ass6} is more straightforward to check.
\end{rem}

\begin{rem}\label{numrem}
	The assumptions in this paper are relatively broad, which cover more types of MVSDEs than those in the existing literature such as \cite{dossup1,dossup2}. For instance, consider the one-dimensional MVSDE
	\begin{equation}\label{exm}
		\begin{split}
			dX_t=(-18X_t^5-X_t^\frac{1}{3}[\mathbb{E}X_t]^4+2)dt+(X_t^2+EX_t)dW_t,
		\end{split}
	\end{equation}		
	which fulfills Assumptions \ref{ass1}-\ref{ass6}, but is not included by any of the works discussed above. We provide the details of verifying Assumptions \ref{ass1}-\ref{ass6} in the Appendix.
\end{rem}

The following theorem establishes the strong well-posedness of (\ref{MV}) under the given assumptions.			
\begin{thm}\label{thm1}
	Assume that Assumptions \ref{ass1}-\ref{ass6} hold with $p \geq \max\{\gamma, 2 + 2l_3, 4l_4\}$. Then, for any $T>0$, the MVSDE (\ref{MV}) admits a unique strong solution $X_t$. Moreover,  for $2\leq\bar{p}\leq p/l_4$, it holds that
	\begin{equation*}
		\sup_{0\leq t\leq T}	\mathbb{E}|X_{t}|^p\leq C_{\xi,p,L_3,T}~~~\text{and}~~~	\mathbb{E}\big[\sup_{0\leq t\leq T}	|X_{t}|^{\bar{p}}\big]\leq C_{\xi,l_{4},L_3,T}.
	\end{equation*}
\end{thm}

For the purpose of approximating the original MVSDE, we consider the following stochastic system of $N$ interacting particles:
\begin{equation}\label{IPS}
	dX_{t}^{i,N}=b(X_{t}^{i,N},\mu_{t}^{X,N})dt+\sigma(X_{t}^{i,N},\mu_{t}^{X,N})dW_{t}^{i},~~ i\in\mathbb{S}_N:=\{1,2,\cdots,N\},
\end{equation}
with the initial value $\xi^i\in\mathbb{L}^p(\mathbb{R}^d)$. Here, $\mu_{t}^{X,N}:=\frac{1}{N}\sum_{j=1}^{N}\delta_{X^{j,N}_t}$ is the empirical measure of $(X^{i,N}_t)_{1\leq i\leq N}$ with $\delta_{X^{j,N}_t}$ being the Dirac measure of $X^{j,N}_t$, and $(W^i,\xi^i)$ are independent copies of $(W,\xi)$.   
%These particles interact with each other.
As the number of particles $N$ goes to infinity, the empirical measure $\mu_{t}^{X,N}$  converges to the law of any mutually independent particles $(X^{i}_t)_{1\leq i\leq N}$ which solves the non-interacting particle system:
\begin{equation}\label{nonIPS}
	dX_{t}^{i}=b(X_{t}^{i},\mathcal{L}(X_{t}^{i}))dt+\sigma(X_{t}^{i},\mathcal{L}(X_{t}^{i}))dW_{t}^{i}, 
\end{equation}
with the initial value $\xi^i\in\mathbb{L}^p(\mathbb{R}^d)$.
The following theorem shows the results of propagation of chaos in finite horizon.

\begin{thm}\label{keythm}
	Assume that Assumptions \ref{ass1}-\ref{ass6} hold with $p \geq \max\{\gamma, 2 + 2l_3, 4l_4\}$. Then, for $2\leq q\leq(\frac{p}{2}+1-l_3)\wedge (\frac{p}{2}+2-2l_4)$ and $2\leq\bar{q}\leq(p/l_2)\wedge q$, we have 
	\begin{equation}\label{A}
		\lim_{N\rightarrow\infty}\sup_{0\leq t\leq T}	\mathbb{E}|X_{t}^{i,N}-X_{t}^{i}|^{q}=0,
	\end{equation} 
	and
	\begin{equation}\label{B}
		\lim_{N\rightarrow\infty}\mathbb{E}\big[\sup_{0\leq t\leq T}|X_{t}^{i,N}-X_{t}^{ i}|^{\bar{q}}\big]=0.
	\end{equation} 
\end{thm}

%	\subsection{Popagation of chaos in infinite horizon}
In the preceding parts, we focused mainly on the well-posedness and propagation of chaos in finite horizon. We now move forward to study the 	propagation of chaos in infinite horizon
by imposing the following conditions:

\begin{ass}\label{assinf2}
	There exist constants $p\geq2$ and $L_5>L_6>0$ such that
	\begin{equation*} 
		2\langle x,b(x,\mu)
		\rangle+(p-1)\|\sigma(x,\mu)\|^2\leq -L_{5}|x|^2+L_6\|\mu\|_2^2,
	\end{equation*} 
	for any $x\in\mathbb{R}^d$ and $\mu\in\mathcal{P}_2(\mathbb{R}^d)$.
\end{ass}
\begin{ass}\label{assinf1}
	For any $R>0$, there exist positive constants $\gamma\geq2$ and $L_{1}$ such that, for any $x,\bar{x}\in\mathbb{R}^d$ with $|x|\vee|\bar{x}|\leq R$ and $\mu$, $\nu\in\mathcal{P}_{\gamma}(\mathbb{R}^d)$,
	\begin{equation*} 
		\begin{split}
			&2	\langle x-\bar{x},b(x,\mu)-b(\bar{x},\nu)\rangle+\|\sigma(x,\mu)-\sigma(\bar{x},\nu)	\|^2\\
			\leq&-h(R)|x-\bar{x}|^2+ \big(g(R)+L_{1}\|\mu\|_\gamma^\gamma+L_{1}\|\nu\|_\gamma^\gamma\big)\big(|x-\bar{x}|^2+\mathbb{W}_{2}^2(\mu,\nu)\big),
		\end{split}
	\end{equation*} 
	where the functions $h:[0,\infty)\rightarrow[0,\infty)$ and $g:[0,\infty)\rightarrow[0,\infty)$ %with $h(R)<\infty, g(R)<\infty$ 
	satisfy 
	\begin{equation} \label{3gR}
		\begin{split}
			h(R)-3g(R)>\left(\gamma/2+1\right)(L_5-L_6),
		\end{split}
	\end{equation} 
	for any $R>0$, and there exists $\lambda>3$ such that
	\begin{equation} \label{3gRr}
		\begin{split}
			\lim_{R\rightarrow\infty}\frac{h(R)}{g(R)}=\lambda,~~~~\lim_{R\rightarrow\infty}g(R)=\infty.
		\end{split}
	\end{equation} 
\end{ass}

The propagation of chaos in infinite horizon is stated as the following theorem.	
\begin{thm}\label{longpoc}
	Assume that Assumptions \ref{ass2}, \ref{assinf2}, \ref{assinf1} hold. Then, we have
	\begin{equation*}
		\lim_{t\rightarrow\infty}	\lim_{N\rightarrow\infty}\frac{1}{N}\sum_{i=1}^{N}\mathbb{E}| X_{t}^{i}-X_{t}^{i,N}|^2=0.
	\end{equation*}
\end{thm}

In the final part of this section, we prepare some conclusions in order to achieve the target of this paper. 
\begin{itemize}
	\item For any $\mu\in\mathcal{P}_{2}(\mathbb{R}^d)$, we have $	\mathbb{W}_2(\mu,\delta_{0})=	\|\mu\|_2$.
	\item For two random variables $X$, $Y$, their distributions are denoted by $\mu$ and $\nu$, respectively. For $p\geq 1$, it holds that
	\begin{equation}\label{W1}
		\mathbb{W}_p(\mu,\nu)\leq\left(  \mathbb{E}\big|X-Y\big|^p\right)^\frac{1}{p}.
	\end{equation}
	\item For the empirical measure  $\tilde{\mu}_{t}^{X,N}:=\frac{1}{N}\sum_{j=1}^{N}\delta_{X^{j}_t}$ constructed by the independent identically distributed samples of process $X_{t}$, for $p\geq 1$, there exists 
	\begin{equation}\label{W3}
		\mathbb{W}_p(\tilde{\mu}_{t}^{X,N},\tilde{\mu}_{t}^{Y,N})\leq	\Big(\frac{1}{N}\sum_{j=1}^{N}|X^{j}_t-Y^{j}_t|^p\Big)^\frac{1}{p}.
	\end{equation}
	\item (Differential version of the Gronwall lemma) Assume that $\theta\in C((0,\infty);R)$ satisfies the differential inequality
	\begin{equation*}
		\theta'(t)\leq a(t)\theta(t)+b(t), 
	\end{equation*}
	for some $a(\cdot)$, $b(\cdot)\in L^{1}(0,\infty)$. Then, $\theta$ satisfies the pointwise estimate 
	\begin{equation}\label{ggron1}
		\theta(t)\leq e^{A(t)}\theta(0)+\int_{0}^{t}b(s)e^{A(t)-A(s)}ds, 
	\end{equation}
	where $A(t)$ is defined by
	$$A(t)=\int_{0}^{t}a(s)ds.$$
\end{itemize}

\section{Proof of Theorem \ref{thm1}}\label{secc3}

\begin{proof}
	The proof is divided into three steps. 
	
	\textbf{Step 1.}
	This step is to construct an Euler-like sequence $\{X_t^{(m)}\}_{m\geq1}$  and prove that it is well-posed. 
	For any integer $m\geq1$ and $T>0$, denote $\Delta_{m}=\frac{T}{m}$ and $t_{k}^{m}=k\Delta_{m}$, $k=0,1,\cdots,m$. The process $X_t^{(m)}$ is then defined sequentially on the time intervals $[0,t_{1}^{m}]$, $(t_{1}^{m},t_{2}^{m}]$, $\cdots$, $(t_{m-1}^{m},T]$. We begin by considering  the following classical SDE on the time interval $[0,t_{1}^{m}]$,
		\begin{equation}\label{MVeu}
			\begin{split}
				dX_{t}^{(m)}=b(X_{t}^{(m)},\mu_{0}^{(m)})dt+\sigma(X_{t}^{(m)},\mu_{0}^{(m)})dW_t,
			\end{split}
		\end{equation}
		%
		%According to (), (\ref{MVeu}) can be transformed into
		with the initial value $X_{0}^{(m)}=\xi$ and $\mu_{0}^{(m)}:=\mathcal{L}\mathcal(\xi)$. Under Assumptions \ref{ass1}-\ref{ass2}, the coefficients of (\ref{MVeu}) satisfy the conditions of $Theorem~3.1.1$ in \cite{liuwei}. Therefore, one can establish the existence and uniqueness of the solution to such equation with super-linear drift and diffusion coefficients on the interval $[0,t_{1}^{m}]$, and the moment boundedness of the solution to (\ref{MVeu}) can also be derived, which is
		\begin{equation*}
			\mathbb{E}\big[\sup_{0\leq t\leq t_{1}^{m}}	|X_{t}^{(m)}|^2\big]\leq C_{L_3,T}(1+\mathbb{E}|\xi|^2).
		\end{equation*}
		Since the proof is analogous to that of Lemma \ref{mbound} below, we omit it here.
		Inductively, we consider the classical SDE for $t\in(t_k^{m},t_{k+1}^{m}]$, $k=1,\cdots, m-1$,
		\begin{equation}\label{MVeeu}
			\begin{split}
				dX_{t}^{(m)}=b(X_{t}^{(m)},\mu_{t_k^{m}}^{(m)})dt+\sigma(X_{t}^{(m)},\mu_{t_k^{m}}^{(m)})dW_t,
			\end{split}
		\end{equation}
		with the initial value $X_{t_k^{m}}^{(m)}$ and $\mu_{t_k^{m}}^{(m)}:=\mathcal{L}(X_{t_k^{m}}^{(m)})$.
		In the same way, the existence and uniqueness of the solution to (\ref{MVeeu}) can also be obtained, and
		\begin{equation*}
			\mathbb{E}\big[\sup_{t_k^{m}\leq t\leq t_{k+1}^{m}}	|X_{t}^{(m)}|^2\big]\leq C_{L_3,T}(1+\mathbb{E}|X_{t_k^{m}}^{(m)}|^2)
		\end{equation*}
		holds. Denote $\lceil t\rceil_m=t_{k}^{m}$ for $t\in(t_k^{m},t_{k+1}^{m}]$. Then for any $t\in[0,T]$, we introduce the following SDE:
		\begin{equation}\label{MVu}
			\begin{split}
				dX_{t}^{(m)}=b(X_{t}^{(m)},\mu_{\lceil t\rceil_m}^{(m)})dt+\sigma(X_{t}^{(m)},\mu_{\lceil t\rceil_m}^{(m)})dW_t,
			\end{split}
		\end{equation}
		with the initial value $\xi$ and $\mu_{\lceil t\rceil_m}^{(m)}:=\mathcal{L}(X_{\lceil t\rceil_m}^{(m)})$.
		% and  initial value $X_{\lfloor t\rfloor_m}^{(m)}$. 
		%	The well-posedness of SDE (\ref{MVu}) can be obtained as well.
		The moment boundedness of the solution to SDE (\ref{MVu}) is obtained in Lemma \ref{mbound} below, and it shows that the bound is irrelevant to $m$. Subsequently, the well-posedness of the Euler-like sequence $\{X_t^{(m)}\}_{m\geq1}$ follows.

		\textbf{Step 2.} 
		This step reveals the existence of the solution to (\ref{MV}).
		We will demonstrate in Lemma \ref{cauchy} below that the sequence $\{X_t^{(m)}\}_{m\geq1}$ is a Cauchy sequence in $L^{q}(\Omega;C([0,T];\mathbb{R}^d))$ for some $q\geq2$, where $L^{q}(\Omega;C([0,T];\mathbb{R}^d))$ denotes the Banach space equipped with the norm $\|X\|_{L^q}:=\left(\mathbb{E}[\sup_{0\leq t\leq T}|X_t|^q]\right)^\frac{1}{q}$. Due to the completeness of $L^{q}(\Omega;C([0,T];\mathbb{R}^d))$, there exists a process $X_{t}\in L^{q}(\Omega;C([0,T];\mathbb{R}^d))$ which is an $\{\mathcal{F}_t\}$-adapted continuous process satisfying 
		\begin{equation*}
			\begin{split}
				&\lim_{m\rightarrow\infty}\mathbb{E}\big[\sup_{0\leq t\leq T}|X_{t}^{(m)}-X_{t}|^2\big]
				=0.
			\end{split}
		\end{equation*}
		Hence, the existence of the solution to (\ref{MV}) is proved.
		
		\textbf{Step 3.}
		This step is to address the uniqueness and moment boundedness of the solution to (\ref{MV}).
		Since the proof follows directly from the procedures of Lemmas \ref{mbound} and \ref{cauchy}, we omit the details.
	\end{proof}

	\begin{lem}\label{mbound}
		Let Assumptions \ref{ass1}-\ref{ass2} hold.\\
		%For any $T>0$, $p_0>\gamma\vee (2 l_{3}+2)\vee (4l_{4})$ and $2\leq\bar{p}_0\leq\frac{p_0}{l_{4}}$, 
		\text{(i)} For $2\leq p_0\leq p$, there exists a positive constant $C_{\xi,p_0,L_3,T}$ such that
		\begin{equation*}
			\sup_{0\leq t\leq T}	\mathbb{E}|X_{t}^{(m)}|^{p_0}\leq C_{\xi,p_0,L_3,T}.
		\end{equation*}
		\text{(ii)} For $2l_4\leq p_0\leq p$ and $2\leq\bar{p}_0\leq p_0/l_4$, there exists a positive constant $C_{\xi,l_{4},\bar{p}_0,L_3,L_4,T}$ such that
		\begin{equation*}
			\mathbb{E}\big[\sup_{0\leq t\leq T}	|X_{t}^{(m)}|^{\bar{p}_0}\big]\leq C_{\xi,l_{4},\bar{p}_0,L_3,L_4,T}.
		\end{equation*}
		Here, $C_{\xi,p_0,L_3,T}$ and $C_{\xi,l_{4},\bar{p}_0,L_3,L_4,T}$ are independent of $m$.
		%	with $\xi\in\mathcal{L}_{p_0}(\mathbb{R}^d)$.
	\end{lem}

	%\subsection{Proof of boundedness}
	\begin{proof}
		For any $t\in[0,T]$, using It\^o's formula leads to 
		\begin{equation*}\label{pp}
			\begin{split}
				&|X_{t}^{(m)}|^{p_0}\\=&|X_{0}^{(m)}|^{p_0}+	\frac{p_0}{2}\int_{0}^{t}|X_{s}^{(m)}|^{{p_0}-2}\Big(2\langle X_{s}^{(m)},b(X_{s}^{(m)},\mu_{\lceil s\rceil_m}^{(m)})\rangle\\&+ ({p_0}-1)\| \sigma(X_{s}^{(m)},\mu_{\lceil s\rceil_m}^{(m)})\| ^2\Big)ds
				+{p_0}\int_{0}^{t}|X_{s}^{(m)}|^{{p_0}-2}\langle X_{s}^{(m)},\sigma(X_{s}^{(m)},\mu_{\lceil s\rceil_m}^{(m)})\rangle dW_s.\\
				%	=:&pQ_1(t)+\frac{p(p-1)}{2}Q_2(t)+pQ_3(t).
			\end{split}
		\end{equation*}
		By Assumption \ref{ass3}, Young's inequality, and H\"older's inequality, we obtain
		\begin{equation*}
			\begin{split}
				\mathbb{E}|X_{t}^{(m)}|^{p_0}
				\leq&\mathbb{E}|\xi|^{p_0}+	\frac{{p_0}L_3}{2}\mathbb{E}\Big[\int_{0}^{t}|X_{s}^{(m)}|^{{p_0}-2}\big( 1+|X_{s}^{(m)}|^2+\|\mu_{\lceil s\rceil_m}^{(m)}\|_{2}^2\big) ds\Big]\\
				\leq&\mathbb{E}|\xi|^{p_0}+L_3 \mathbb{E}\Big[\int_{0}^{t}\big( 1+(\frac{3p_0}{2}-2)|X_{s}^{(m)}|^{{p_0}}+\mathbb{E}|X_{\lceil s\rceil_m}^{(m)}|^{{p_0}}\big) ds\Big].\\
			\end{split}
		\end{equation*}
		Therefore, for any $t\in[0,T]$,
		\begin{equation*}
			\begin{split}
				\sup_{0\leq t\leq T}\mathbb{E}|X_{t}^{(m)}|^{p_0}\leq&\mathbb{E}|\xi|^{p_0}+L_3 T+(\frac{3p_0}{2}-1)L_3 \int_{0}^{T}\sup_{0\leq s\leq t}\mathbb{E}|X_{s}^{(m)}|^{{p_0}}dt	.
			\end{split}
		\end{equation*}
		By using the Gronwall inequality, we see that
		\begin{equation}\label{supE}
			\begin{split}
				\sup_{0\leq t\leq T}\mathbb{E}|X_{t}^{(m)}|^{p_0}\leq&(\mathbb{E}|\xi|^{p_0}+L_3 T)\exp\left((\frac{3p_0}{2}-1)L_3 T\right).
			\end{split}
		\end{equation}
		In what follows, by applying BDG's inequality, Young's inequality, and Assumptions \ref{ass3}-\ref{ass2}, we arrive at
		\begin{equation*}
			\begin{split}
				&\mathbb{E}\big[\sup_{0\leq s\leq t}|X_{s}^{(m)}|^{\bar{p}_0}\big]\\
				\leq&\mathbb{E}|\xi|^{\bar{p}_0}+	\frac{{\bar{p}_0}L_3}{2}\mathbb{E}\Big[\int_{0}^{t}|X_{s}^{(m)}|^{{\bar{p}_0}-2}\big( 1+|X_{s}^{(m)}|^2+\|\mu_{\lceil s\rceil_m}^{(m)}\|_{2}^2\big) ds\Big]\\&+{\bar{p}_0}4\sqrt{2}\mathbb{E}\Big(\int_{0}^{t}|X_{s}^{(m)}|^{2{\bar{p}_0}-2}\|\sigma(X_{s}^{(m)},\mu_{\lceil s\rceil_m}^{(m)})\|^2ds\Big)^\frac{1}{2}\\
				\leq&\mathbb{E}|\xi|^{\bar{p}_0}+L_3 \mathbb{E}\Big[\int_{0}^{t}\big( 1+(\frac{{3\bar{p}_0}}{2}-2)|X_{s}^{(m)}|^{{\bar{p}_0}}+\mathbb{E}|X_{\lceil s\rceil_m}^{(m)}|^{\bar{p}_0}\big) ds\Big]\\&+\frac{1}{2}\mathbb{E}\big[\sup_{0\leq s\leq t}|X_{s}^{(m)}|^{\bar{p}_0}\big]+K_{\bar{p}_0,T}\mathbb{E}\Big[\int_{0}^{t}\|\sigma(X_{s}^{(m)},\mu_{\lceil s\rceil_m}^{(m)}\|^{\bar{p}_0} ds\Big]\\
				%			\end{split}
			%		\end{equation*}
		%		\begin{equation*}
			%			\begin{split}
				\leq&\mathbb{E}|\xi|^{\bar{p}_0}+L_3 \mathbb{E}\Big[\int_{0}^{t}\big( 1+(\frac{{3\bar{p}_0}}{2}-2)|X_{s}^{(m)}|^{{\bar{p}_0}}+\mathbb{E}|X_{\lceil s\rceil_m}^{(m)}|^{\bar{p}_0}\big) ds\Big]\\
				&+
				\frac{1}{2}\mathbb{E}\big[\sup_{0\leq s\leq t}|X_{s}^{(m)}|^{\bar{p}_0}\big]+{K}_{{\bar{p}_0},T}3^{\bar{p}_0-1}L_4^{\bar{p}_0}\mathbb{E}\Big[\int_{0}^{t}\big( 1+|X_{s}^{(m)}|^{{\bar{p}_0}l_{4}}+\mathbb{E}|X_{\lceil s\rceil_m}^{(m)}|^{\bar{p}_0}\big) ds\Big],\\
			\end{split}
		\end{equation*}
		with $K_{\bar{p}_0,T}:=4\sqrt{2}\big(8\sqrt{2}({\bar{p}_0}-1)\big)^{{\bar{p}_0}-1}T^\frac{{\bar{p}_0}-2}{2}$. Then,
		\begin{equation*}
			\begin{split}
				&\mathbb{E}\big[\sup_{0\leq t\leq T}|X_{t}^{(m)}|^{\bar{p}_0}\big]\\
				\leq&2\mathbb{E}|\xi|^{\bar{p}_0}+	2(L_3 +{K}_{{\bar{p}_0},T}3^{\bar{p}_0-1}L_4^{\bar{p_0}})T\\&+((3\bar{p}_0-3)L_3 +2{K}_{{\bar{p}_0},T}3^{\bar{p}_0-1}L_4^{\bar{p}_0})\int_{0}^{T}\mathbb{E}\big[\sup_{0\leq s\leq t}|X_{s}^{(m)}|^{\bar{p}_0}\big] dt\\&+2{K}_{{\bar{p}_0},T}3^{\bar{p}_0-1}L_4^{\bar{p}_0}\int_{0}^{T}\mathbb{E}|X_{t}^{(m)}|^{\bar{p}_0l_{4}}dt.
			\end{split}
		\end{equation*}
		By applying the Gronwall inequality and (\ref{supE}), one has
		\begin{equation*}\label{Esup}
			\begin{split}
				&\mathbb{E}\big[\sup_{0\leq t\leq T}|X_{t}^{(m)}|^{\bar{p}_0}\big]\\
				\leq&\big(2\mathbb{E}|\xi|^{\bar{p}_0}+	2(L_3 +{K}_{{\bar{p}_0},T}3^{\bar{p}_0-1}L_4^{\bar{p}_0})T+2{K}_{{\bar{p}_0},T}3^{\bar{p}_0-1}L_4^{\bar{p}_0}T\sup_{0\leq t\leq T}\mathbb{E}|X_{t}^{(m)}|^{\bar{p}_0l_{4}}\big)\\&\cdot\exp\left((3\bar{p}_0-3)L_3 +2{K}_{{\bar{p}_0},T}3^{\bar{p}_0-1}L_4^{\bar{p}_0}\right),
			\end{split}
		\end{equation*}
		as required.
	\end{proof}
	
	\begin{lem}\label{cauchy}
		Assume that Assumptions \ref{ass1}-\ref{ass6} hold with $p \geq \max\{\gamma, 2 + 2l_3, 4l_4\}$. For $2\leq q\leq(\frac{p}{2}+1-l_3)\wedge (\frac{p}{2}+2-2l_4)$ and $2\leq\bar{q}\leq(p/l_2)\wedge q$, we have
		\begin{equation*}
			\begin{split}
				&\sup_{0\leq t\leq T}\mathbb{E}\big[|X_{t}^{(m)}-X_{t}^{ (n)}|^{q}\big]
				\rightarrow0,
			\end{split}
		\end{equation*}
		and
		\begin{equation*}
			\begin{split}
				&\mathbb{E}\big[\sup_{0\leq t\leq T}|X_{t}^{(m)}-X_{t}^{ (n)}|^{\bar{q}}\big]
				\rightarrow0,
			\end{split}
		\end{equation*}
		as $m,n\rightarrow\infty$. %$\xi\in\mathcal{L}_{q}(\mathbb{R}^d)$.
	\end{lem}
	%and ${q}\geq2$
	\begin{proof}
		For any $t\in[0,T]$, by It\^o's formula, we obtain
		\begin{equation*}\label{only}
			\begin{split}
				&|X_{t}^{(m)}-X_{t}^{ (n)}|^{q}\\=&	{q}\int_{0}^{t}|X_{s}^{(m)}-X_{s}^{ (n)}|^{{q}-2}\langle X_{s}^{(m)}-X_{s}^{ (n)},b(X_{s}^{(m)},\mu^{(m)}_{\lceil s\rceil_m})-b(X_{s}^{(n)},\mu^{(n)}_{\lceil s\rceil_n})\rangle ds\\&
				+\frac{{q}({q}-1)}{2}\int_{0}^{t}|X_{s}^{(m)}-X_{s}^{ (n)}|^{{q}-2}\| \sigma(X_{s}^{(m)},\mu^{(m)}_{\lceil s\rceil_m})-\sigma(X_{s}^{(n)},\mu^{(n)}_{\lceil s\rceil_n})\| ^2ds\\&
				+{q}\int_{0}^{t}|X_{s}^{(m)}-X_{s}^{ (n)}|^{{q}-2}\langle X_{s}^{(m)}-X_{s}^{ (n)},\sigma(X_{s}^{(m)},\mu^{(m)}_{\lceil s\rceil_m})-\sigma(X_{s}^{(n)},\mu^{(n)}_{\lceil s\rceil_n})\rangle dW_s\\
				=:&I_{1}(t)+I_{2}(t)+I_{3}(t).
			\end{split}
		\end{equation*}

		By using a similar proof of exponential integrability in Section \ref{aa6.1}, under Assumption \ref{ass6}, one can show that there exists a non-decreasing function $f:[0,\infty)\rightarrow[0,\infty)$ with $\mathbb{E}e^{f(|\xi|)}<\infty$ satisfying
		\begin{equation}\label{cheb}
			\sup_{0\leq t\leq T}\mathbb{E}[\exp({f(|X_t^{(m)}|)})]<\infty.
		\end{equation}
		Define 
		\begin{equation*}
			{	\Omega}^{m,n}(R)=\big\{\omega\in\Omega:|X^{(m)}_t|>R,~|X^{(n)}_t|> R\big\}
		\end{equation*}
		for a sufficiently large $R$. For any $q^{*}\geq2$ and the given function $f(\cdot)$, using the Chebyshev inequality gives that 
		\begin{equation}\label{cheb02}
			\begin{split}
				&\mathbb{P}({	\Omega}^{m,n}(R))\\\leq& \exp({-2q^{*}f(R)})\left(\sup_{0\leq t\leq T}\mathbb{E}\big[\exp(2q^{*}{f(|X^{(m)}_t|)})\big]+\sup_{0\leq t\leq T}\mathbb{E}\big[\exp(2q^{*}{f(|X^{(n)}_t|)})\big]\right)\\\leq& C_{\xi,q^{*},\alpha,\beta,L_3,T,f,\bar{f}}\exp({-2q^{*}f(R)}).
			\end{split}
		\end{equation}
		It is noteworthy that the estimate of this probability is independent of both $m$ and $n$.				By Assumptions \ref{ass1}, \ref{ass2}, and (\ref{W1}), we have
		\begin{equation}\label{chang}
			\begin{split}
				&\mathbb{E}\big[|X_{t}^{(m)}-X_{t}^{ (n)}|^{q}\big]\\=&\frac{{q}}{2}\mathbb{E}\Big[\int_{0}^{t}|X_{s}^{(m)}-X_{s}^{ (n)}|^{{q}-2}\big(2\langle X_{s}^{(m)}-X_{s}^{ (n)},b(X_{s}^{(m)},\mu^{(m)}_{\lceil s\rceil_m})-b(X_{s}^{(n)},\mu^{(n)}_{\lceil s\rceil_n})\rangle ds\\&~~~~+({q}-1)\| \sigma(X_{s}^{(m)},\mu^{(m)}_{\lceil s\rceil_m})-\sigma(X_{s}^{(n)},\mu^{(n)}_{\lceil s\rceil_n})\| ^2\big)ds\Big]\\\leq&\frac{{q}}{2}\mathbb{E}\Big[\Big(\int_{0}^{t}|X_{s}^{(m)}-X_{s}^{ (n)}|^{{q}-2}\big(2\langle X_{s}^{(m)}-X_{s}^{ (n)},b(X_{s}^{(m)},\mu^{(m)}_{\lceil s\rceil_m})-b(X_{s}^{(n)},\mu^{(n)}_{\lceil s\rceil_n})\rangle 
				\\&~~~~+({q}-1)\| \sigma(X_{s}^{(m)},\mu^{(m)}_{\lceil s\rceil_m})-\sigma(X_{s}^{(n)},\mu^{(n)}_{\lceil s\rceil_n})\| ^2\big)ds\Big)\cdot\mathbb{I}_{\Omega\setminus{	\Omega}^{m,n}(R)}\Big]\\
				&+{q}\mathbb{E}\Big[\Big(\int_{0}^{t}|X_{s}^{(m)}-X_{s}^{ (n)}|^{{q}-1}|b(X_{s}^{(m)},\mu^{(m)}_{\lceil s\rceil_m})-b(X_{s}^{(n)},\mu^{(n)}_{\lceil s\rceil_n})| ds\Big)\cdot\mathbb{I}_{{	\Omega}^{m,n}(R)}\Big]\\
				&+\frac{{q}({q}-1)}{2}\mathbb{E}\Big[\Big(\int_{0}^{t}|X_{s}^{(m)}-X_{s}^{ (n)}|^{{q}-2}\|\sigma(X_{s}^{(m)},\mu^{(m)}_{\lceil s\rceil_m})\\&~~~~~~~~~~~~~~~~~-\sigma(X_{s}^{(n)},\mu^{(n)}_{\lceil s\rceil_n})\|^2 ds\Big)\cdot\mathbb{I}_{{	\Omega}^{m,n}(R)}\Big]\\
				\leq&\frac{{q}}{2}\mathbb{E}\Big[\int_{0}^{t}|X_{s}^{(m)}-X_{s}^{ (n)}|^{{q}-2}\big(L(R)+L_{1}\|\mu^{(m)}_{\lceil s\rceil_m}\|_{\gamma}^{\gamma}+L_{1}\|\mu^{(n)}_{\lceil s\rceil_n}\|_{\gamma}^{\gamma}\big)\big(|X_{s}^{(m)}-X_{s}^{ (n)}|^{2}\\&~~~~+\mathbb{W}_{2}^{2}(\mu^{(m)}_{\lceil s\rceil_m},\mu^{(n)}_{\lceil s\rceil_n})\big)ds\Big]\\
				&+{q}L_{4}\mathbb{E}\Big[\Big(\int_{0}^{t}\big(|X_{s}^{(m)}|+|X_{s}^{ (n)}|\big)^{{q}-1}\big(2+|X_{s}^{(m)}|^{l_{3}}+\|\mu^{(m)}_{\lceil s\rceil_m}\|_{l_{3}}^{l_{3}}+|X_{s}^{(n)}|^{l_{3}}\\&~~~~~~~~~~+\|\mu^{(n)}_{\lceil s\rceil_n}\|_{l_{3}}^{l_{3}} \big)ds\Big)\cdot\mathbb{I}_{{	\Omega}^{m,n}(R)}\Big]\\
				&+3{q}({q}-1)L_{4}^2\mathbb{E}\Big[\Big(\int_{0}^{t}\big(|X_{s}^{(m)}|+|X_{s}^{ (n)}|\big)^{{q}-2}\big(2+|X_{s}^{(m)}|^{2l_{4}}+\|\mu^{(n)}_{\lceil s\rceil_n}\|_{2}^2\\&~~~~~~~~~~~~~~~~~~~~~+|X_{s}^{(n)}|^{2l_{4}}+\|\mu^{(n)}_{\lceil s\rceil_n}\|_2^2 \big) ds\Big)\cdot\mathbb{I}_{{	\Omega}^{m,n}(R)}\Big]\\
				\leq&\frac{{q}}{2}\mathbb{E}\Big[\int_{0}^{t}|X_{s}^{(m)}-X_{s}^{ (n)}|^{{q}-2}\big(L(R)+L_{1}\mathbb{E}|X_{\lceil s\rceil}^{(m)}|^{\gamma}+L_{1}\mathbb{E}|X_{\lceil s\rceil}^{(n)}|^{\gamma}\big)\big(|X_{s}^{(m)}-X_{s}^{ (n)}|^{2}\\&~~~~+3\mathbb{W}_{2}^{2}(\mu^{(m)}_{\lceil s\rceil_m},\mu^{(m)}_{s})+3\mathbb{W}_{2}^{2}(\mu^{(m)}_{s},\mu^{(n)}_{s})+3\mathbb{W}_{2}^{2}(\mu^{(n)}_{s},\mu^{(n)}_{\lceil s\rceil_n})\big)ds\Big]\\
				&+{q}L_{4}\mathbb{E}\Big[\Big(\int_{0}^{t}\big(|X_{s}^{(m)}|+|X_{s}^{ (n)}|\big)^{{q}-1}\big(2+|X_{s}^{(m)}|^{l_{3}}+\mathbb{E}|X^{(m)}_{\lceil s\rceil_m}|^{l_{3}}+|X_{s}^{(n)}|^{l_{3}}\\&~~~~~~~~~~~+\mathbb{E}|X^{(n)}_{\lceil s\rceil_n}|^{l_{3}} \big)ds\Big)\cdot\mathbb{I}_{{	\Omega}^{m,n}(R)}\Big]\\
				&+3{q}({q}-1)L_{4}^2\mathbb{E}\Big[\Big(\int_{0}^{t}\big(|X_{s}^{(m)}|+|X_{s}^{ (n)}|\big)^{{q}-2}\big(2+|X_{s}^{(m)}|^{2l_{4}}+\mathbb{E}|X^{(m)}_{\lceil s\rceil_m}|^2\\&~~~~~~~~~~~~~~~~~~~~~+|X_{s}^{(n)}|^{2l_{4}}+\mathbb{E}|X^{(n)}_{\lceil s\rceil_n}|^2 \big) ds\Big)\cdot\mathbb{I}_{{	\Omega}^{m,n}(R)}\Big].\\
			\end{split}
		\end{equation}

		Furthermore, the boundedness of the solution to (\ref{MVu}) implies that for any $q\geq 2$,
		%An application of the BDG inequality and H\"oler inequality with Assumptions \ref{ass1}-\ref{ass2} gives that
		\begin{equation}\label{mmnn}
			\begin{split}
				&\mathbb{E}|X_t^{(m)}-X_{\lceil t\rceil_m}^{(m)}|^q\\
				\leq&2^{q-1}\Big(\mathbb{E}\Big|\int_{\lceil t\rceil_m}^{t}b(X_{s}^{(m)},\mu_{\lceil s\rceil_m}^{(m)})ds\Big|^q+\mathbb{E}\Big|\int_{\lceil t\rceil_m}^{t}\sigma(X_{s}^{(m)},\mu_{\lceil s\rceil_m}^{(m)})dW_s\Big|^q\Big)\\
				\leq&2^{q-1}\Big(\Delta_{m}^{q-1}\int_{\lceil t\rceil_m}^{t}\mathbb{E}|b(X_{s}^{(m)},\mu_{\lceil s\rceil_m}^{(m)})|^qds+\mathbb{E}\big(\int_{\lceil t\rceil_m}^{t}\|\sigma(X_{s}^{(m)},\mu_{\lceil s\rceil_m}^{(m)})\|^2ds\big)^\frac{q}{2}\Big)\\
				\leq&2^{q-1}\Big(\Delta_{m}^{q-1}\int_{\lceil t\rceil_m}^{t}\mathbb{E}|b(X_{s}^{(m)},\mu_{\lceil s\rceil_m}^{(m)})|^qds+\Delta_{m}^{\frac{q}{2}-1}\int_{\lceil t\rceil_m}^{t}\|\sigma(X_{s}^{(m)},\mu_{\lceil s\rceil_m}^{(m)})\|^qds\Big)\\
				\leq&C_{\xi,l_{3},l_{4},q,L_{3},L_4,t}\Delta_{m}^\frac{q}{2},
			\end{split}
		\end{equation}
		where we have utilized the H\"older inequality and Assumption \ref{ass2}.

		Hence,  exploiting H\"older's inequality, Young's inequality, (\ref{W1}), and (\ref{mmnn}), we obtain
		\begin{equation}\label{001}
			\begin{split}
				&\sup_{0\leq t\leq T}\mathbb{E}\big[|X_{t}^{(m)}-X_{t}^{ (n)}|^{q}\big]\\
				\leq&\big(L(R)+C_{\xi,\gamma,L_{1},L_3,T}\big)\int_{0}^{T}\Big((5q-6)\sup_{0\leq s\leq t}\mathbb{E}\big[|X_{s}^{(m)}-X_{s}^{ (n)}|^{q}]+3\mathbb{E}\big[|X_{t}^{(m)}-X_{\lceil t\rceil_m}^{ (m)}|^{q}\big]\\&~~~~~~~~~~~~~~~~~~~~~~~~~~~~~+3\mathbb{E}\big[|X_{t}^{(n)}-X_{\lceil t\rceil_n}^{ (n)}|^{q}\big]\Big)dt\\
				&+C_{q,L_{4},T}\Big(1+\sup_{0\leq t\leq T}\mathbb{E}\big[|X_{t}^{(m)}|^{2({q}+l_{3}-1)}\big]+\sup_{0\leq t\leq T}\mathbb{E}\big[|X_{t}^{(n)}|^{2({q}+l_{3}-1)}\big]\Big)^\frac{1}{2}\big(\mathbb{P}({	\Omega}^{m,n}(R))\big)^\frac{1}{2}\\
				&+C_{q,L_{4},T}\Big(1+\sup_{0\leq t\leq T}\mathbb{E}\big[|X_{t}^{(m)}|^{2({q}+2l_{4}-2)}\big]+\sup_{0\leq t\leq T}\mathbb{E}\big[|X_{t}^{(n)}|^{2({q}+2l_{4}-2)}\big]\Big)^\frac{1}{2}\big(\mathbb{P}({	\Omega}^{m,n}(R))\big)^\frac{1}{2}\\
				\leq&(5{q}-6)\big(L(R)+C_{\xi,\gamma,L_{1},L_3,T}\big)\int_{0}^{T}\sup_{0\leq s\leq t}\mathbb{E}\big[|X_{s}^{(m)}-X_{s}^{ (n)}|^{q}\big]dt\\&+\bar{L}(R)(\Delta_{m}^\frac{{q}}{2}+\Delta_{n}^{\frac{{q}}{2}})
				+C_{\xi,l_{3},l_{4},q,L_3,L_{4},T}\big(\mathbb{P}({	\Omega}^{m,n}(R))\big)^\frac{1}{2},
			\end{split}
		\end{equation}
		where 
		$\bar{L}(\cdot): [0,\infty)\rightarrow[0,\infty)$ is a function associated with $L(\cdot)$.
		Thanks to Gronwall's inequality and (\ref{cheb02}), we derive that
		\begin{equation}\label{0009}
			\begin{split}
				&\sup_{0\leq t\leq T}\mathbb{E}\big[|X_{t}^{(m)}-X_{t}^{ (n)}|^{q}\big]\\
				\leq&C_{\xi,l_{3},l_{4},q,L_{3},L_4,T}\big(\mathbb{P}({	\Omega}^{m,n}(R))\big)^\frac{1}{2}\exp\left({(5q-6)(L(R)+C_{\xi,\gamma,L_{1},L_3,T})}\right)\\&+\bar{L}(R)(\Delta_{m}^\frac{q}{2}+\Delta_{n}^{\frac{q}{2}})\exp\left({(5q-6)(L(R)+C_{\xi,\gamma,L_{1},L_3,T})}\right)\\\leq&C_{\xi,\gamma,\alpha,\beta,l_{3},l_{4},q,L_{1},L_{3},L_4,T,f,\bar{f}}\exp\left({-(5q-6)(f(R)-L(R))}\right)\\&+\bar{L}(R)(\Delta_{m}^\frac{q}{2}+\Delta_{n}^{\frac{q}{2}})\exp\left({(5q-6)(L(R)+C_{\xi,\gamma,L_{1},L_3,T})}\right).
			\end{split}
		\end{equation}
		By (\ref{R}), it means that for any $\eta_1>0$, there exists an $R_{\eta_1}$ large enough such that, when $R\geq R_{\eta_1}$, the first term on the right-hand side of the inequality (\ref{0009}) is less than $\eta_1/2$.
		As $m$, $n\rightarrow\infty$, $\Delta_{m}$ and $\Delta_{n}$ tend to 0. Then, there exists an integer $U_{0}$ such that for all $m,n\geq U_0$, 
		%on both sides of (\ref{0009}), so that $\Delta_{m}$ and $\Delta_{n}$ can be made sufficiently small for}
	\begin{equation*}
		\begin{split}
			& \bar{L}(R_{\eta_1})(\Delta_{m}^\frac{q}{2}+\Delta_{n}^{\frac{q}{2}})\exp\left({(5q-6)(L(R_{\eta_1})+C_{\xi,\gamma,L_{1},L_3,T})}\right)\leq \eta_1/2.
		\end{split}
	\end{equation*}
	That is, for any $\eta_1>0$, one can first choose a constant $R_{\eta_1}>0$, and then find an integer $U_{0}$ such that for all $R\geq R_{\eta_1}$ and $m,n\geq U_0$,
	\begin{equation*}
		\begin{split}
			&\sup_{0\leq t\leq T}\mathbb{E}\big[|X_{t}^{(m)}-X_{t}^{ (n)}|^{q}\big]
			\leq\eta_1.
		\end{split}
	\end{equation*}

	We proceed to prove the second statement of Lemma \ref{cauchy}.				
	It is straightforward to verify that
	\begin{equation}\label{supmn}
		\begin{split}
			&\mathbb{E}\big[\sup_{0\leq s\leq t}|X_{s}^{(m)}-X_{s}^{ (n)}|^{\bar{q}}\big]\\\leq&
			\frac{\bar{q}}{2}\mathbb{E}\Big[\int_{0}^{t}|X_{s}^{(m)}-X_{s}^{ (n)}|^{\bar{q}-2}\big(L(R)+L_{1}\|\mu^{(m)}_{\lceil s\rceil_m}\|_{\gamma}^{\gamma}+L_{1}\|\mu^{(n)}_{\lceil s\rceil_n}\|_{\gamma}^{\gamma}\big)\big(|X_{s}^{(m)}-X_{s}^{ (n)}|^{2}\\&~~~~+\mathbb{W}_{2}^{2}(\mu^{(m)}_{\lceil s\rceil_m},\mu^{(n)}_{\lceil s\rceil_n})\big)ds\Big]\\+&	\bar{q}\mathbb{E}\Big[\Big(\int_{0}^{t}|X_{s}^{(m)}-X_{s}^{ (n)}|^{\bar{q}-1}|b(X_{s}^{(m)},\mu^{(m)}_{\lceil s\rceil_m})-b(X_{s}^{(n)},\mu^{(n)}_{\lceil s\rceil_n})| ds\Big)\cdot\mathbb{I}_{{	\Omega}^{m,n}(R)}\Big]\\
			+&\frac{\bar{q}({\bar{q}}-1)}{2}\mathbb{E}\Big[\Big(\int_{0}^{t}|X_{s}^{(m)}-X_{s}^{ (n)}|^{\bar{q}-2}\|\sigma(X_{s}^{(m)},\mu^{(m)}_{\lceil s\rceil_m})\\&~~~~~~~~~~~~~-\sigma(X_{s}^{(n)},\mu^{(n)}_{\lceil s\rceil_n})\|^2 ds\Big)\cdot\mathbb{I}_{{	\Omega}^{m,n}(R)}\Big]\\
			+&{\bar{q}}\mathbb{E}\Big[\sup_{0\leq s\leq t}\int_{0}^{s}|X_{u}^{(m)}-X_{u}^{ (n)}|^{\bar{q}-2}\langle X_{u}^{(m)}-X_{u}^{ (n)},\sigma(X_{u}^{(m)},\mu^{(m)}_{\lceil u\rceil_m})\\&~~~-\sigma(X_{u}^{(n)},\mu^{(n)}_{\lceil u\rceil_n})\rangle dW_u\Big].
		\end{split}
	\end{equation}
	By using BDG's inequality, Young's inequality, H\"older's inequality, Assumption \ref{ass4}, (\ref{W1}), and Lemma \ref{mbound} we have
	\begin{equation*}
		\begin{split}
			&\bar{q}	\mathbb{E}\big[\sup_{0\leq s\leq t}\int_{0}^{s}|X_{u}^{(m)}-X_{u}^{ (n)}|^{\bar{q}-2}\langle X_{u}^{(m)}-X_{u}^{ (n)},\sigma(X_{u}^{(m)},\mu^{(m)}_{\lceil u\rceil_m})\sigma(X_{u}^{(n)},\mu^{(n)}_{\lceil u\rceil_n})\rangle dW_u\big]\\\leq&
			\bar{q}(32/\bar{q})^{\bar{q}/2}	\mathbb{E}\left(\int_{0}^{t}|X_{s}^{(m)}-X_{s}^{ (n)}|^{2\bar{q}-2}\|\sigma(X_{s}^{(m)},\mu^{(m)}_{\lceil s\rceil_m})-\sigma(X_{s}^{(n)},\mu^{(n)}_{\lceil s\rceil_n})\|^2 ds\right)^\frac{1}{2}\\\leq&
			\frac{1}{2}\mathbb{E}\big[\sup_{0\leq s\leq t}|X_{s}^{(m)}-X_{s}^{ (n)}|^{\bar{q}}\big]+C_{\bar{q},T}\mathbb{E}\left[\int_{0}^{t}\|\sigma(X_{s}^{(m)},\mu^{(m)}_{\lceil s\rceil_m})-\sigma(X_{s}^{(n)},\mu^{(n)}_{\lceil s\rceil_n})\|^{\bar{q} }ds\right]\\\leq&
			\frac{1}{2}\mathbb{E}\big[\sup_{0\leq s\leq t}|X_{s}^{(m)}-X_{s}^{ (n)}|^{\bar{q}}\big]+C_{\bar{q},T}\mathbb{E}\left[\int_{0}^{t}\|\sigma(X_{s}^{(m)},\mu^{(m)}_{\lceil s\rceil_m})-\sigma(X_{s}^{(n)},\mu^{(m)}_{\lceil s\rceil_m})\|^{\bar{q} }ds\right]\\
			&+C_{\bar{q},T}\mathbb{E}\left[\int_{0}^{t}\|\sigma(X_{s}^{(n)},\mu^{(m)}_{\lceil s\rceil_m})-\sigma(X_{s}^{(n)},\mu^{(n)}_{\lceil s\rceil_n})\|^{\bar{q} }ds\right]\\
							\end{split}
		\end{equation*}
	\begin{equation}\label{BDG}
\begin{split}
			\leq&
			\frac{1}{2}\mathbb{E}\big[\sup_{0\leq s\leq t}|X_{s}^{(m)}-X_{s}^{ (n)}|^{\bar{q}}\big]\\&+C_{\bar{q},L_2,T}\mathbb{E}\left[\int_{0}^{t}\left(1+|X_{s}^{(m)}|^{\bar{q}l_2}+|X_{s}^{(n)}|^{\bar{q}l_2}\right)|X_{s}^{(m)}-X_{s}^{ (n)}|^{\bar{q}}ds\right]\\&+C_{\bar{q},L_2,T}\mathbb{E}\left[\int_{0}^{t}\left(\mathbb{W}_{2}^{\bar{q}}(\mu^{(m)}_{\lceil s\rceil_m},\mu^{(m)}_{s})+\mathbb{W}_{2}^{\bar{q}}(\mu^{(m)}_{s},\mu^{(n)}_{s})+\mathbb{W}_{2}^{\bar{q}}(\mu^{(n)}_{s},\mu^{(n)}_{\lceil s\rceil_n})\right)ds\right]\\
			\leq&
			\frac{1}{2}\mathbb{E}\big[\sup_{0\leq s\leq t}|X_{s}^{(m)}-X_{s}^{ (n)}|^{\bar{q}}\big]\\&+C_{\bar{q},L_2,T}\int_{0}^{t}\left(1+\mathbb{E}|X_{s}^{(m)}|^\frac{\bar{q}l_2(\bar{q}+\varepsilon_1)}{\varepsilon_1}+\mathbb{E}|X_{s}^{(n)}|^\frac{\bar{q}l_2(\bar{q}+\varepsilon_1)}{\varepsilon_1}\right)^\frac{\varepsilon_1}{\bar{q}+\varepsilon_1}\\&~~~~~~~~~~~~~~\cdot\left(\mathbb{E}|X_{s}^{(m)}-X_{s}^{ (n)}|^{\bar{q}+\varepsilon_1}\right)^\frac{\bar{q}}{\bar{q}+\varepsilon_1}ds\\
			&+C_{\bar{q},L_2,T}\int_{0}^{t}\Big(\mathbb{E}\big[|X_{s}^{(m)}-X_{\lceil s\rceil_m}^{ (m)}|^{\bar{q}}\big]+\mathbb{E}\big[|X_{s}^{(m)}-X_{s}^{ (n)}|^{\bar{q}}]+\mathbb{E}\big[|X_{s}^{(n)}-X_{\lceil s\rceil_n}^{ (n)}|^{\bar{q}}\big]\Big)ds
		\end{split}
	\end{equation}
	with an arbitrary positive constant ${\varepsilon_1}$ satisfying $\bar{q}^2l_{2}/(p-\bar{q}l_{2})\leq\varepsilon_1\leq q-\bar{q}$.
	Substituting (\ref{BDG}) into (\ref{supmn}) and following a similar proof as in (\ref{chang}) and (\ref{001}), we conclude that
	\begin{equation*}
		\begin{split}
			&\mathbb{E}\big[\sup_{0\leq t\leq T}|X_{t}^{(m)}-X_{t}^{ (n)}|^{\bar{q}}\big]\\\leq&	(10\bar{q}-12)\big(L(R)+C_{\xi,\gamma,\bar{q},L_{1},L_2,L_3,T}\big)\int_{0}^{T}\mathbb{E}\big[\sup_{0\leq s\leq t}|X_{s}^{(m)}-X_{s}^{ (n)}|^{\bar{q}}\big]dt\\&+\hat{L}(R)(\Delta_{m}^\frac{\bar{q}}{2}+\Delta_{n}^{\frac{\bar{q}}{2}})
			+C_{\xi,l_{3},l_{4},\bar{q},L_{3},L_4,T}\big(\mathbb{P}({	\Omega}^{m,n}(R))\big)^\frac{1}{2}\\&+C_{\xi,l_2,\bar{q},\varepsilon_1,L_2,L_3,T}\left(\sup_{0\leq t\leq T}\mathbb{E}|X_{t}^{(m)}-X_{t}^{ (n)}|^{\bar{q}+\varepsilon_1}\right)^\frac{\bar{q}}{\bar{q}+\varepsilon_1},
		\end{split}
	\end{equation*}
	where %$q\leq(\frac{p_0}{2}+1-l_{3})\wedge(\frac{p_0}{2}+2-2l_{4})$ and 
	$\hat{L}(\cdot): [0,\infty)\rightarrow[0,\infty)$ is a function related to $L(\cdot)$.
	According to (\ref{0009}), we deduce that
	\begin{equation*}
		\begin{split}
			&\mathbb{E}\big[\sup_{0\leq t\leq T}|X_{t}^{(m)}-X_{t}^{ (n)}|^{\bar{q}}\big]\\\leq&	(10\bar{q}-12)\big(L(R)+C_{\xi,\gamma,\bar{q},L_{1},L_2,L_3,T}\big)\int_{0}^{T}\mathbb{E}\big[\sup_{0\leq s\leq t}|X_{s}^{(m)}-X_{s}^{ (n)}|^{\bar{q}}\big]dt\\&+\hat{L}(R)(\Delta_{m}^\frac{\bar{q}}{2}+\Delta_{n}^{\frac{\bar{q}}{2}})
			+C_{\xi,l_{3},l_{4},\bar{q},L_{3},L_4,T}\big(\mathbb{P}({	\Omega}^{m,n}(R))\big)^\frac{1}{2}\\&+C_{\xi,\gamma,\alpha,\beta,l_{2},l_{3},l_{4},\varepsilon_1,\bar{q},L_{1},L_{2},L_3,L_{4},T,f,\bar{f}}\exp\left({-(5\bar{q}-6)(f(R)-L(R))}\right)\\&+(2\bar{L}(R))^\frac{\bar{q}}{\bar{q}+\varepsilon_1}(\Delta_{m}^\frac{\bar{q}}{2}+\Delta_{n}^{\frac{\bar{q}}{2}})\exp\left({5\bar{q}L(R)+C_{\xi,\gamma,L_{1},L_3,T}}\right).
		\end{split}
	\end{equation*}
	By (\ref{cheb02}) and the Gronwall inequality, we have
	\begin{equation}\label{unigron}
		\begin{split}
			&	\mathbb{E}\big[\sup_{0\leq t\leq T}|X_{t}^{(m)}-X_{t}^{ (n)}|^{\bar{q}}\big]\\	
			\leq&C_{\xi,\gamma,\bar{q},\alpha,\beta,l_{3},l_{4},L_{1},L_{2},L_3,L_{4},T,f,\bar{f}}\exp\left({-(10\bar{q}-12)(f(R)-L(R))}\right)	\\&+C_{\xi,\gamma,\alpha,\beta,l_{2},l_{3},l_{4},\varepsilon_1,\bar{q},L_{1},L_{2},L_3,L_{4},T,f,\bar{f}}\exp\left({-(5\bar{q}-6)(f(R)-3L(R))}\right)\\&
			+C_{\xi,\gamma,\bar{q},L_{1},L_2,L_3,T}\left(\hat{L}(R)(\Delta_{m}^\frac{\bar{q}}{2}+\Delta_{n}^{\frac{\bar{q}}{2}})+(2\bar{L}(R))^\frac{\bar{q}}{\bar{q}+\varepsilon_1}(\Delta_{m}^\frac{\bar{q}}{2}+\Delta_{n}^{\frac{\bar{q}}{2}})\right)\cdot\\&~~~~\exp\left((15\bar{q}-12)f(R)\right).
		\end{split}
	\end{equation}
	According to (\ref{R}), for any $\eta_2>0$, choose an $R_{\eta_2}$ sufficiently large so that the first two terms on the right-hand side of (\ref{unigron}) are less than $\eta_2/2$, whenever $R\geq R_{\eta_2}$. Similarly, one may further choose an integer $\bar{U}_0$ such that for all $m,n\geq \bar{U}_0$, the last term on the right-hand side of (\ref{unigron}) is less than $\eta_2/2$. That is,
	\begin{equation*}
		\begin{split}
			\lim_{m,n\rightarrow\infty}	\mathbb{E}\big[\sup_{0\leq t\leq T}|X_{t}^{(m)}-X_{t}^{ (n)}|^{\bar{q}}\big]	\rightarrow0.
		\end{split}
	\end{equation*}
	The proof is therefore completed.
\end{proof}

\section{ Proof of Theorem \ref{keythm}}\label{secc4}

This section aims to demonstrate the propagation of chaos theory under locally monotone condition w.r.t. the state variable and measure argument. Firstly, the well-posedness of the interacting particle system is established in the following lemma.

\begin{lem}\label{keylem}
	Let Assumptions \ref{ass1}-\ref{ass2} hold with $p\geq2l_4\vee\gamma$. The system (\ref{IPS}) admits a unique strong solution. Moreover, for $i\in\mathbb{S}_N$ and $2\leq\bar{p}\leq p/l_4$, it holds that
	\begin{equation*}
		\sup_{0\leq t\leq T}	\mathbb{E}|X_{t}^{i,N}|^{p}\leq C_{\xi,p,L_3,T},
	\end{equation*}
	and
	\begin{equation*}
		\mathbb{E}\big[\sup_{0\leq t\leq T}	|X_{t}^{i,N}|^{\bar{p}}\big]\leq C_{\xi,l_{4},L_3,T}.
	\end{equation*}
	%			with $\xi\in\mathcal{L}^{q_2}(\mathbb{R}^d)$ and $i\in\mathbb{S}_N$.
\end{lem}
\begin{proof}
	For $\mathbf{x}:=(x_1^\top,x_2^\top,\cdots,x_N^\top)^\top\in\mathbb{R}^{dN}$, $x_i\in\mathbb{R}^d$, define
	\begin{equation*}
		\mathbf{b}(\mathbf{x})=\left(b(x_1,\frac{1}{N}\sum_{j=1}^{N}\delta_{x_j})^\top,b(x_2,\frac{1}{N}\sum_{j=1}^{N}\delta_{x_j})^\top,\cdots,b(x_N,\frac{1}{N}\sum_{j=1}^{N}\delta_{x_j})^\top\right)^\top,
	\end{equation*}
	and 
	\begin{equation*}
		\left.\bm{\sigma}(\mathbf{x})=\left(\begin{array}{cccc}\sigma\left(x_1,\frac1N\sum_{j=1}^N\delta_{x_j}\right)&\bm{0}&\cdots&\bm{0}\\\bm{0}&\sigma\left(x_2,\frac1N\sum_{j=1}^N\delta_{x_j}\right)&\cdots&\bm{0}\\\vdots&\vdots&\ddots&\vdots\\\bm{0}&\bm{0}&\cdots&\sigma\left(x_N,\frac1N\sum_{j=1}^N\delta_{x_j}\right)\end{array}\right.\right),
	\end{equation*}
	where $\bm{0}$ denotes the $d\times m_1$ null matrix. 
	Denote	$\mathbf{X}_{t}^{N}=({X_t^{1,N}}^\top,{X_t^{2,N}}^\top,\cdots,{X_t^{N,N}}^\top)^\top$.
	Thus, (\ref{IPS}) can be transformed into
	\begin{equation}\label{MVv}
		d\mathbf{X}_{t}^{N}=\mathbf{b}(\mathbf{X}_{t}^{N})dt+\bm{\sigma}(\mathbf{X}_{t}^{N})d\mathbf{W}_{t}^{N}, 
	\end{equation}
	with the initial value $\xi^{N}:=\left({\xi^{1}}^\top,{\xi^{2}}^\top,\cdots,{\xi^{N}}^\top\right)^\top$ and $$\mathbf{W}_{t}^{N}:= \left({W^{1}_t}^\top,{W^{2}_t}^\top,\cdots,{W^{N}_t}^\top\right)^\top,$$ which is an $m_1N$ dimensional Brownian motion.
	Applying (\ref{W3}) yields that
	\begin{equation*} 
		\mathbb{W}_{2}^2(\frac{1}{N}\sum_{j=1}^{N}\delta_{x_j},\frac{1}{N}\sum_{j=1}^{N}\delta_{\bar{x}_j})\leq\frac{1}{N}\sum_{j=1}^{N}|x_j-\bar{x}_j|^2.
	\end{equation*} 
	By Assumption \ref{ass1}, it can be verified that, for any $|\mathbf{x}|\vee|\mathbf{\bar{x}}|\leq R$,
	\begin{equation*}
		\begin{split}
			&2	\langle \mathbf{x}-\mathbf{\bar{x}},\mathbf{b}(\mathbf{x})-\mathbf{b}(\mathbf{\bar{x}})\rangle+(p-1)\|\bm{\sigma}(\mathbf{x})-\bm{\sigma}(\mathbf{\bar{x}})\|^2\\
			=&\sum_{i=1}^{N}\Big(2	\langle x_i-\bar{x}_i,b(x_i,\frac{1}{N}\sum_{j=1}^{N}\delta_{x_j})-b(\bar{x}_i,\frac{1}{N}\sum_{j=1}^{N}\delta_{\bar{x}_j})\rangle\\&~~+(p-1)\|\sigma(x_i,\frac{1}{N}\sum_{j=1}^{N}\delta_{x_j})-\sigma(\bar{x}_i,\frac{1}{N}\sum_{j=1}^{N}\delta_{\bar{x}_j})\|^2\Big)\\
			\leq&\sum_{i=1}^{N} \Big(L(R)+L_{1}\frac{1}{N}\sum_{j=1}^{N}|x_j|^\gamma+L_{1}\frac{1}{N}\sum_{j=1}^{N}|\bar{x}_j|^\gamma\Big)\big(|x_i-\bar{x}_i|^2+\frac{1}{N}\sum_{j=1}^{N}|x_j-\bar{x}_j|^2\big)\\
		\end{split}
	\end{equation*} 
	\begin{equation}\label{IPSxin} 
		\begin{split}
			\leq& \big(L(R)+L_{1}|\mathbf{x}|^\gamma+L_{1}|\mathbf{\bar{x}}|^\gamma\big)\big(\sum_{i=1}^{N}|x_i-\bar{x}_i|^2+\sum_{j=1}^{N}|x_j-\bar{x}_j|^2\big)\\
			=:& \tilde{L}(R)|\mathbf{x}-\mathbf{\bar{x}}|^2.
		\end{split}
	\end{equation} 
	Furthermore, using Assumption \ref{ass2} implies 
	\begin{equation}\label{cub} 
		\begin{split}
			|\mathbf{b}(\mathbf{x})|
			=\sum_{i=1}^{N}\Big| b(x_i,\frac{1}{N}\sum_{j=1}^{N}\delta_{x_j})\Big|
			\leq &L_{4}\sum_{i=1}^{N} \Big(1+|x_i|^{l_{3}}+\frac{1}{N}\sum_{j=1}^{N}|x_j|^{l_{3}}\Big)\\
			\leq& 2L_{4}N \big(1+|\mathbf{x}|^{l_{3}}\big),
		\end{split}
	\end{equation} 
	and 
	\begin{equation}\label{cus} 
		\begin{split}
			\|\bm{\sigma}(\mathbf{x})\|^2
			=\sum_{i=1}^{N}\Big\| \sigma(x_i,\frac{1}{N}\sum_{j=1}^{N}\delta_{x_j})\Big\|^2
			\leq &L^2_4\sum_{i=1}^{N} \Big(1+|x_i|^{l_{4}}+\Big(\frac{1}{N}\sum_{j=1}^{N}|x_j|^{2}\Big)^\frac{1}{2}\Big)^2\\
			\leq& 3L_{4} ^2\sum_{i=1}^{N} \Big(1+|x_i|^{2l_{4}}+\frac{1}{N}\sum_{j=1}^{N}|x_j|^{2}\Big)\\
			\leq& 6L_{4} ^2 N\Big(1+|\mathbf{x}|^{2l_{4}}\Big).
		\end{split}
	\end{equation} 
	By \eqref{IPSxin}, \eqref{cus} and \cite[Theorem~3.1.1]{liuwei}, the existence and uniqueness of the solution to \eqref{MVv} can be established.
	Once the well-posedness of (\eqref{IPS}) is verified, the moment boundedness of its solution follows from Assumption~\ref{ass3}, together with the fact that the particles are identically distributed. The proof is analogous to that of Theorem~\ref{thm1}. It is worth noting that the derived bounds are irrelevant to particle number $N$. So we state the result without proof. 
\end{proof}

We are now in the position to prove Theorem \ref{keythm}.

\begin{proof}
	%						The deviation of system (\ref{MV0}) and (\ref{projectMV}) is represented by $$\Delta_{t}^{n}=X_{t}^{i,N}-X_{t}^{ n}.$$
	For any $i\in\mathbb{S}_N$ and large enough $R>0$, define 
	\begin{equation*}
		{	\Omega}^{i}(R)=\big\{\omega\in\Omega:|X^{i}_t|>R,~|X^{i,N}_t|>R\big\}.
	\end{equation*}
	
	According to Assumption \ref{ass6}, it could be found that there exists a non-decreasing function $f:[0,\infty)\rightarrow[0,\infty)$ with $\mathbb{E}e^{f(|\xi^i|)}<\infty$ satisfying
	\begin{equation}\label{cheb2}
		\sup_{0\leq t\leq T}\mathbb{E}[\exp(f({|X_t^{i,N}|)})]<\infty,~~~\sup_{0\leq t\leq T}\mathbb{E}[\exp({f(|X_t^{i}|)})]<\infty.
	\end{equation}
	For any $q^{*}\geq2$ and the given function $f(R)$, applying the Chebyshev inequality and (\ref{cheb2}) gives that 
	\begin{equation}\label{pccheb}
		\begin{split}
			&	\mathbb{P}(	{	\Omega}^{i}(R))\\\leq &\exp({-2q^{*}f(R)})\left(\sup_{0\leq t\leq T}\mathbb{E}\big[\exp({2q^{*}f(|X^{i,N}_t|)})\big]+\sup_{0\leq t\leq T}\mathbb{E}\big[\exp({2q^{*}f(|X^{i}_t|)})\big]\right)\\\leq &C_{\xi,q^{*},\alpha,\beta,L_3,T,f,\bar{f}} \exp({-2q^{*}f(R)}).
		\end{split}
	\end{equation}
	Using It\^o's formula leads to
	\begin{equation}\label{poc}
		\begin{split}
			&	|X_{t}^{i,N}-X_{t}^{ i}|^{q}\\=&	{q}\int_{0}^{t}	|X_{s}^{i,N}-X_{s}^{i}|^{{q}-2}\langle X_{s}^{i,N}-X_{s}^{i},b(X_{s}^{i,N},\mu_{s}^{X,N})-b(X_{s}^{i},\mu_{s}^i)\rangle ds\\
			&+\frac{{q}({q}-1)}{2}\int_{0}^{t}|X_{s}^{i,N}-X_{s}^{i}|^{{q}-2}\|\sigma(X_{s}^{i,N},\mu_{s}^{X,N})-\sigma(X_{s}^{i},\mu_{s}^i)\|^2ds\\&
			+{q}\int_{0}^{t}|X_{s}^{i,N}-X_{s}^{i}|^{{q}-2}\langle X_{s}^{i,N}-X_{s}^{i},\sigma(X_{s}^{i,N},\mu_{s}^{X,N})-\sigma(X_{s}^{i},\mu_{s}^i)\rangle  dW_s^i.\\
			%	=:&pJ_1^{n}(t)+\frac{p(p-1)}{2}J_2^{n}(t)+pJ_3^{n}(t).
		\end{split}
	\end{equation}
	Then, from Assumptions \ref{ass1}, \ref{ass2}, and Young's inequlity, we derive 
	\begin{equation*}\label{poc1}
		\begin{split}
			&	\mathbb{E}	|X_{t}^{i,N}-X_{t}^{i}|^{q}\\
			=&\frac{q}{2}\mathbb{E}\Big[\int_{0}^{t}|X_{s}^{i,N}-X_{s}^{i}|^{{q}-2}\big(2\langle X_{s}^{i,N}-X_{s}^{i},b(X_{s}^{i,N},\mu_{s}^{X,N})-b(X_{s}^{i},\mu_{s}^i)\rangle ds\\&+({q}-1)\|\sigma(X_{s}^{i,N},\mu_{s}^{X,N})-\sigma(X_{s}^{i},\mu_{s}^i)\|^2\big)ds\Big]\\
				\end{split}
		\end{equation*}
				\begin{equation*}
				\begin{split}
			\leq&\frac{{q}}{2}\mathbb{E}\Big[\Big(\int_{0}^{t}|X_{s}^{i,N}-X_{s}^{i}|^{{q}-2}\big(2\langle X_{s}^{i,N}-X_{s}^{i},b(X_{s}^{i,N},\mu_{s}^{X,N})-b(X_{s}^{i},\mu_{s}^i)\rangle ds\\&+({q}-1)\|\sigma(X_{s}^{i,N},\mu_{s}^{X,N})-\sigma(X_{s}^{i},\mu_{s}^i)\|^2\big)ds\Big)\cdot\mathbb{I}_{\Omega\setminus{	\Omega}^{i}(R)}\Big]\\
			&+{q}\mathbb{E}\Big[\Big(\int_{0}^{t}|X_{s}^{i,N}-X_{s}^{i}|^{{q}-1}|b(X_{s}^{i,N},\mu_{s}^{X,N})-b(X_{s}^{i},\mu_{s}^i)| ds\Big)\cdot\mathbb{I}_{{	\Omega}^{i}(R)}\Big]\\
			&+\frac{{q}({q}-1)}{2}\mathbb{E}\Big[\Big(\int_{0}^{t}|X_{s}^{i,N}-X_{s}^{i}|^{{q}-2}\|\sigma(X_{s}^{i,N},\mu_{s}^{X,N})-\sigma(X_{s}^{i},\mu_{s}^i)\|^2 ds\Big)\cdot\mathbb{I}_{{	\Omega}^{i}(R)}\Big]\\
			\leq&\frac{{q}}{2}\mathbb{E}\Big[\int_{0}^{t} |X_{s}^{i,N}-X_{s}^{i}|^{{q}-2}\big(L(R)+L_{1}\mathbb{E}|X_{s}^{i,N}|^{\gamma}+L_{1}\mathbb{E}|X_{s}^{i}|^{\gamma}\big)\\&~~~~~\cdot\big(|X_{s}^{i,N}-X_{s}^{i}|^{2}+\mathbb{W}_{2}^{2}(\mu^{X,N}_{s},\mu^{i}_{s})\big)ds\Big]\\	
			&+{q}L_{4}\mathbb{E}\Big[\Big(\int_{0}^{t}\big(|X_{s}^{i,N}|+|X_{s}^{i}|\big)^{{q}-1}\\&~~~~~~~~~~~\cdot\big(2+|X_{s}^{i,N}|^{l_{3}}+\mathbb{E}|X_{s}^{i,N}|^{l_{3}}+|X_{s}^{i}|^{l_{3}}+\mathbb{E}|X_{s}^{i}|^{l_{3}} \big)ds\Big)\cdot\mathbb{I}_{{	\Omega}^{i}(R)}\Big]\\
			&+3{q}({q}-1)L_{4}^2\mathbb{E}\Big[\Big(\int_{0}^{t}\big(|X_{s}^{i,N}|+|X_{s}^{i}|\big)^{{q}-2}\big(2+|X_{s}^{i,N}|^{2l_{4}}+\mathbb{E}|X_{s}^{i,N}|^{2}\\&~~~~~~~~~~~~~~~~~~~~~+|X_{s}^{i}|^{2l_{4}}+\mathbb{E}|X_{s}^{i}|^{2} \big) ds\Big)\cdot\mathbb{I}_{{	\Omega}^{i}(R)}\Big].\\
		\end{split}
	\end{equation*}
	By the notation $\tilde{\mu}^{X,N}_{t}:=\frac{1}{N}\sum_{j=1}^{N}\delta_{X^{j}_t}$, Young's inequality, H\"older's inequality, and the fact that $(X^{j,N}-X^{j})_{j\in\mathbb{S}_n}$ are identically distributed, we arrive at
	\begin{equation*}
		\begin{split}
			&	\mathbb{E}	|X_{t}^{i,N}-X_{t}^{i}|^{q}\\
			\leq&\big(L(R)+C_{\xi,\gamma,L_{1},L_3,T}\big)\frac{{q}}{2}\mathbb{E}\Big[\int_{0}^{t} |X_{s}^{i,N}-X_{s}^{i}|^{{q}}ds\Big]\\	&+
			\big(L(R)+C_{\xi,\gamma,L_{1},L_3,T}\big){q}\mathbb{E}\Big[\int_{0}^{t} |X_{s}^{i,N}-X_{s}^{i}|^{{q-2}}\big(\mathbb{W}_{2}^{2}(\mu^{X,N}_{s},\tilde{\mu}^{X,N}_{s})\\&~~~~~~~~~~~~~~~~~~~~~~~~~~~~~~~~~~+\mathbb{W}_{2}^{2}(\tilde{\mu}^{X,N}_{s},\mu^{i}_{s})\big)ds\Big]\\	&+{q}L_{4}\mathbb{E}\Big[\Big(\int_{0}^{t}\big(|X_{s}^{i,N}|+|X_{s}^{i}|\big)^{{q}-1}\\&~~~~~~~~~~\cdot\big(2+|X_{s}^{i,N}|^{l_{3}}+\mathbb{E}|X_{s}^{i,N}|^{l_{3}}+|X_{s}^{i}|^{l_{3}}+\mathbb{E}|X_{s}^{i}|^{l_{3}} \big)ds\Big)\cdot\mathbb{I}_{{	\Omega}^{i}(R)}\Big]\\
			&+3{q}({q}-1)L_{4}^2\mathbb{E}\Big[\Big(\int_{0}^{t}\big(|X_{s}^{i,N}|+|X_{s}^{i}|\big)^{{q}-2}\big(2+|X_{s}^{i,N}|^{2l_{4}}+\mathbb{E}|X_{s}^{i,N}|^{2}\\&~~~~~~~~~~~~~~~~~~~~~+|X_{s}^{i}|^{2l_{4}}+\mathbb{E}|X_{s}^{i}|^{2} \big) ds\Big)\cdot\mathbb{I}_{{	\Omega}^{i}(R)}\Big]\\
			\leq&(L(R)+C_{\xi,\gamma,L_{1},L_3,T})(\frac{5q}{2}-4)\int_{0}^{t}\mathbb{E} |X_{s}^{i,N}-X_{s}^{i}|^{{q}}ds\\&+2(L(R)+C_{\xi,\gamma,L_{1},L_3,T})\int_{0}^{t}\mathbb{E} \big[\mathbb{W}_{2}^{q}(\mu^{X,N}_{s},\tilde{\mu}^{X,N}_{s})\big]ds\\
					\end{split}
		\end{equation*}
	\begin{equation*}
\begin{split}
			&	+2(L(R)+C_{\xi,\gamma,L_{1},L_3,T})\int_{0}^{t}\mathbb{E} \big[\mathbb{W}_{2}^{q}(\tilde{\mu}^{X,N}_{s},\mu^{i}_{s})\big]ds\\&+C_{\xi,l_{3},l_{4},q,L_{3},L_4,T}\big(\mathbb{P}({	\Omega}^{i}(R))\big)^\frac{1}{2}\\
			\leq&(L(R)+C_{\xi,\gamma,L_{1},L_3,T})(\frac{5{q}}{2}-2)\int_{0}^{t}\mathbb{E} |X_{s}^{i,N}-X_{s}^{i}|^{{q}}ds\\&+2(L(R)+C_{\xi,\gamma,L_{1},L_3,T})\int_{0}^{T}\mathbb{E}\big[ \mathbb{W}_{2}^{q}(\tilde{\mu}^{X,N}_{s},\mu^{i}_{s})\big]ds
			+C_{\xi,l_{3},l_{4},q,L_{3},L_4,T}\big(\mathbb{P}({	\Omega}^{i}(R))\big)^\frac{1}{2},\\
		\end{split}
	\end{equation*}
	where we have used (\ref{W3}), Theorem \ref{thm1}, and Lemma \ref{keylem}.

	By using Gronwall's inequality and (\ref{pccheb}), one can see that 
	\begin{equation*}
		\begin{split}
			&	\sup_{0\leq t\leq T}\mathbb{E}	|X_{t}^{i,N}-X_{t}^{ i}|^{q}\\
			\leq&\big(2L(R)+C_{\xi,\gamma,L_{1},L_3,T}\big)\exp\big((\frac{5{q}}{2}-2)(L(R)+C_{\xi,\gamma,L_{1},L_3,T})\big)\int_{0}^{T}\mathbb{E}\big[ \mathbb{W}_{2}^{{q}}(\tilde{\mu}^{X,N}_{s},\mu^{i}_{s})\big]ds\\&+C_{\xi,l_{3},l_{4},q,L_{3},L_4,T}\exp\big((\frac{5{q}}{2}-2)(L(R)+C_{\xi,\gamma,L_{1},L_3,T})\big)\big(\mathbb{P}({	\Omega}^{i}(R))\big)^\frac{1}{2}\\
			\leq&C_{\xi,\gamma,q,L_{1},L_3,T}\exp\left({\frac{5q}{2}L(R)}\right)	\int_{0}^{T}\mathbb{E}\big[ \mathbb{W}_{2}^{{q}}(\tilde{\mu}^{X,N}_{s},\mu^{i}_{s})\big]ds\\&+C_{\xi,\gamma,\alpha,\beta,l_{3},l_{4},q,L_{1},L_{3},L_4,T,f,\bar{f}}\exp\left({-(\frac{5{q}}{2}-2)(f(R)-L(R))}\right).
		\end{split}
	\end{equation*}
	Since $R$ is independent of $N$, letting $N\rightarrow\infty$, the Lebesgue dominated convergence theorem with the result
	$$	\lim_{N\rightarrow\infty}\mathbb{E}\big[ \mathbb{W}_{2}^{q}(\tilde{\mu}^{X,N}_{s},\mu^{i}_{s})\big]=0$$in \cite[Theorem 5.8]{ 2018} implies that
	\begin{equation*}
		\begin{split}
			&\lim_{N\rightarrow\infty}\sup_{0\leq t\leq T}	\mathbb{E}	|X_{t}^{i,N}-X_{t}^{ i}|^{q}\\
			\leq&C_{\xi,\gamma,\alpha,\beta,l_{3},l_{4},q,L_{1},L_{3},L_4,T,f,\bar{f}}\exp\left({-(\frac{5{q}}{2}-2)(f(R)-L(R))}\right).
		\end{split}
	\end{equation*}
	We complete the proof by letting $R\rightarrow\infty$. Accordingly, \eqref{B} can be established by applying the same lines of proof as in Lemma \ref{cauchy}.
\end{proof}

\section{Proof of Theorem \ref{longpoc}}\label{secc5}
In this section, we intend to complete the proof of Theorem \ref{longpoc}.
The following lemma is presented as a preliminary step.
%	In this section, we intend to complete the proof of Theorem \ref{longpoc}. We first reveal the stability properties of particle systems in the following lemma. 
The proof is omitted since it is similar to that in \cite{zhuoqi}.

\begin{lem}\label{expo}
	Let Assumptions  \ref{assinf2} and \ref{assinf1} hold. For $i\in\mathbb{S}_N$ and any $a_1, a_2\in(0,p(L_5-L_6)/2]$, the solution $X_{t}^{i,N}$ to (\ref{IPS}) and the solution $X_{t}^{i}$ to  (\ref{nonIPS}) are exponentially stable in $p$-$th$ moment sense, i.e.,
	\begin{equation*}
		\lim_{t\rightarrow\infty}\frac{1}{t}\log\big(	\mathbb{E}|X_{t}^{i,N}|^{p}\big)\leq -a_1,
	\end{equation*}
	and
	\begin{equation*}
		\lim_{t\rightarrow\infty}\frac{1}{t}\log\big(	\mathbb{E}|X_{t}^{i}|^{p}\big)\leq -a_2.
	\end{equation*}
	
\end{lem}
We now establish the proof of Theorem~\ref{longpoc}.
\begin{proof}
	We first take the expectation on both sides of (\ref{poc}) and then differentiate w.r.t. $t$. By Assumptions \ref{assinf2} and \ref{assinf1}, it follows that
	\begin{equation}\label{buti}
		\begin{split}
			&\frac{d}{dt}\mathbb{E}|X_{t}^{i,N}-X_{t}^{i}|^2\\=&	2\mathbb{E}\langle X_{t}^{i,N}-X_{t}^{i},b(X_{t}^{i,N},\mu_{t}^{X,N})-b(X_{t}^{i},\mu_{t}^{i})\rangle 
			+\mathbb{E}\|\sigma(X_{t}^{i,N},\mu_{t}^{X,N})-\sigma(X_{t}^{i},\mu_{t}^{i})\|^2\\
			\leq&-h(R)\mathbb{E}|X_{t}^{i,N}-X_{t}^{i}|^2\\&+\big(g(R)+L_{1}\mathbb{E}|X_{t}^{i,N}|^{\gamma}+L_{1}\mathbb{E}|X_{t}^{i}|^{\gamma}\big)\big(\mathbb{E}|X_{t}^{i,N}-X_{t}^{i}|^2+\mathbb{E}[\mathbb{W}_2^2(\mu_{t}^{X,N},\mu_{t}^{i})] \big)\\
			\leq&\left(-h(R)+3g(R)+C_{\xi,\gamma,L_{1}}e^{-at}\right)\mathbb{E}|X_{t}^{i,N}-X_{t}^{i}|^2\\&+\left(2g(R)+C_{\xi,\gamma,L_{1}}e^{-at}\right)\mathbb{E}[\mathbb{W}_2^2(\tilde{\mu_{t}}^{X,N},\mu_{t}^{i})],
			%+C_{\xi,l_{3},l_{4},L_{4}}e^{-\frac{1}{2}at}\big(\mathbb{P}({	\Omega}^{i}(R))\big)^\frac{1}{2},
		\end{split}
	\end{equation}
	with $a:=\gamma(L_5-L_6)/2$.
	Next, summing (\ref{buti}) over $j = 1$ to $N$ and dividing by $N$ yields that
	\begin{equation*}
		\begin{split}
			&\frac{d}{dt}\frac{1}{N}\sum_{j=1}^{N}\mathbb{E}|X_{t}^{j,N}-X_{t}^{j}|^2\\\leq&	\left(-h(R)+3g(R)+C_{\xi,\gamma,L_{1}}e^{-at}\right)\frac{1}{N}\sum_{j=1}^{N}\mathbb{E}| X_{t}^{j,N}-X_{t}^{j}|^2\\
			&+\left(2g(R)+C_{\xi,\gamma,L_{1}}e^{-at}\right)\mathbb{E}[\mathbb{W}_2^2(\tilde{\mu_{t}}^{X,N},\mu_{t}^{i})].
		\end{split}
	\end{equation*}
	By taking $p=2$ in \cite[Theorem~1]{guillin}, we obtain that for any $\tilde{q}>2$, 
	\begin{equation*}\label{ppc}
		\begin{split}
			\mathbb{E}[\mathbb{W}_2^2(\tilde{\mu_{t}}^{X,N},\mu_{t}^{i})]\leq C_{d,\tilde{q}}\left(\mathbb{E}|X_t|^{\tilde{q}}\right)^\frac{2}{\tilde{q}}\Phi(N),
		\end{split}
	\end{equation*}
	where $C_{d,\tilde{q}}$ is a constant which only depends on ${d,\tilde{q}}$, and $$\Phi(N):=\begin{cases}N^{-\frac{1}{2}}+N^{-\frac{\tilde{q}-2}{\tilde{q}}},& d<4 ~\text{and}~\tilde{q}\neq 4,\\N^{-\frac{1}{2}}\log(1+N)+N^{-\frac{\tilde{q}-2}{\tilde{q}}},&d=4~ \text{and}~ \tilde{q}\neq 4,\\N^{-\frac{2}{d}}+N^{-\frac{\tilde{q}-2}{\tilde{q}}},&d>4 ~\text{and}~ \tilde{q}\neq \frac{d}{d-2}.\end{cases}$$
	By Lemma \ref{expo}, it can be concluded that
	\begin{equation}\label{lllo}
		\begin{split}
			&\frac{d}{dt}\frac{1}{N}\sum_{j=1}^{N}\mathbb{E}|X_{t}^{j,N}-X_{t}^{j}|^2\\\leq&	\left(-h(R)+3g(R)+C_{\xi,\gamma,L_{1}}e^{-at}\right)\frac{1}{N}\sum_{j=1}^{N}\mathbb{E}| X_{t}^{j,N}-X_{t}^{j}|^2\\
			&+\left(2g(R)+C_{\xi,\gamma,L_{1}}e^{-at}\right)C_{d,\tilde{q}}\left(\mathbb{E}|\xi|^{\tilde{q}}\right)^\frac{2}{\tilde{q}}e^{-bt}\Phi(N),
			%+C_{\xi,l_{3},l_{4},L_{4}}e^{-\frac{1}{2}at}\frac{1}{N}\sum_{j=1}^{N}\big(\mathbb{P}({	\Omega}^{j}(R))\big)^\frac{1}{2},\\
		\end{split}
	\end{equation}
	with $b:=L_5-L_6$.
	Denote $$\theta(t)=\frac{1}{N}\sum_{j=1}^{N}\mathbb{E}|X_{t}^{j,N}-X_{t}^{j}|^2.$$
	From (\ref{lllo}), we have
	\begin{equation*}
		\begin{split}
			{\theta}'(t)\leq&	\left(-h(R)+3g(R)+C_{\xi,\gamma,L_{1}}e^{-at}\right)\theta(t)\\
			&+\left(2g(R)+C_{\xi,\gamma,L_{1}}e^{-at}\right)C_{d,\tilde{q}}\left(\mathbb{E}|\xi|^{\tilde{q}}\right)^\frac{2}{\tilde{q}}e^{-bt}\Phi(N),\\&
			%+C_{\xi,l_{3},l_{4},L_{4}}e^{-\frac{1}{2}at}\frac{1}{N}\sum_{j=1}^{N}\big(\mathbb{P}({	\Omega}^{j}(R))\big)^\frac{1}{2},
		\end{split}
	\end{equation*}
	with $\theta(0)=0$. By the differential version of the Gronwall inequality (\ref{ggron1}),
	it holds that 
	$$A(t)=\left(-h(R)+3g(R)\right)t+\frac{C_{\xi,\gamma,L_{1}}}{a}(1-e^{-at}),$$and
	\begin{equation*}\label{pob}
		\begin{split}
			&	{\theta}(t)\\\leq& e^{A(t)}2g(R)C_{\xi,d,\tilde{q}}\Phi(N)\int_{0}^{t}e^{-A(s)-bs} ds+e^{A(t)}C_{\xi,\gamma,L_{1},d,\tilde{q}}\Phi(N)\int_{0}^{t}e^{-A(s)-as-bs} ds\\
			%	e^{A(t)}2g(R)C_{d,q}\left(\mathbb{E}|\xi|^q\right)^\frac{2}{q}\Phi(N)\int_{0}^{t}e^{-A(s)-bs} ds\\&+e^{A(t)}C_{\xi,\gamma,L_{1},d,q}\left(\mathbb{E}|\xi|^q\right)^\frac{2}{q}\Phi(N)\int_{0}^{t}e^{-A(s)-as-bs} ds\\
			%+e^{A(t)}C_{\xi,l_{3},l_{4},L_{4}}\frac{1}{N}\sum_{j=1}^{N}\big(\mathbb{P}({	\Omega}^{j}(R))\big)^\frac{1}{2}\int_{	0}^{t}e^{-A(s)-\frac{1}{2}as} ds\\
			\leq&\frac{2g(R)C_{\xi,\gamma,L_{1},d,\tilde{q},a}\varphi(N)}{h(R)-3g(R)-b}\left(e^{-bt}-e^{(-h(R)+3g(R))t}\right)\\&+\frac{C_{\xi,\gamma,L_{1},d,\tilde{q},a}\varphi(N)}{h(R)-3g(R)-a-b}\left(e^{-(a+b)t}-e^{(-h(R)+3g(R))t}\right)
			\\
			%&+\frac{C_{\xi,l_{3},l_{4},L_{1},L_{4},a}\frac{1}{N}\sum_{j=1}^{N}\big(\mathbb{P}({	\Omega}^{j}(R))\big)^\frac{1}{2}}{h(R)-3g(R)-\frac{1}{2}a}\left(e^{-\frac{1}{2}at}-e^{(-h(R)+3g(R))t}\right)\\
			\leq&\frac{2g(R)C_{\xi,\gamma,L_{1},d,\tilde{q},a}\varphi(N)e^{-bt}}{h(R)-3g(R)-b}+\frac{C_{\xi,\gamma,L_{1},d,q,a}\varphi(N)e^{-(a+b)t}}{h(R)-3g(R)-a-b}.\\
		\end{split}
	\end{equation*}
	%	where (\ref{3gR}) has been used.
	It should be emphasized that $R$ and $N$ are independent. Then, we let $R\rightarrow\infty$. From (\ref{3gRr}), it follows that
	\begin{equation*}
		\begin{split}
			\lim_{R\rightarrow\infty}	\frac{\varphi(N)e^{-(a+b)t}}{h(R)-3g(R)-a-b}=0,
		\end{split}
	\end{equation*}
	and
	\begin{equation*}
		\begin{split}
			{\theta}(t)\leq& 
			C_{\xi,\gamma,L_{1},d,\tilde{q},a,\lambda}\varphi(N)e^{-bt},
		\end{split}
	\end{equation*}
	which implies
	$$\lim_{t\rightarrow\infty}\lim_{N\rightarrow\infty}\frac{1}{N}\sum_{j=1}^{N}\mathbb{E}|X_{t}^{j,N}-X_{t}^{j}|^2=0.$$
\end{proof}

\section{Numerical simulation}
Consider the one-dimensional MVSDE 
\begin{equation}\label{exms}
	\begin{split}
		dX_t=(-18X_t^5-X_t^\frac{1}{3}[\mathbb{E}X_t]^4+2)dt+(X_t^2+EX_t)dW_t,
	\end{split}
\end{equation}		
with initial value $\xi\sim\mathcal{N}(0,1)$, where $\mathcal{N}(\cdot,\cdot)$ is normal distribution.
Under a given time grid, the strong pathwise propagation of chaos is evaluated by
$$\text{error}:=\sqrt{\frac{1}{N_l}\sum_{j=1}^{N_l}\left(X_T^{j,N_l}-\tilde{X}_T^{j,N_l}\right)^2},$$ at time $T$.
The particle system $\{\tilde{X}_T^{i,N_l}\}_{i\in\mathbb{S}_{N_l}}$ is formed by splitting the collection of Brownian motions that drive the system $\{X_T^{i,N_l}\}_{i\in\mathbb{S}_{N_l}}$  into two equal groups, each consisting of  $N_l/2$ particles. Consequently, the subsystem $\{\tilde{X}_T^{i,N_l,(1)}\}_{i\in\{1,\cdots,N_l/2\}}$ and $\{\tilde{X}_T^{i,N_l,(2)}\}_{i\in\{N_l/2+1,\cdots,N_l\}}$ are driven by the Brownian motions $(W^i)_{i\in\{1,\cdots,N_l/2\}}$ and $(W^i)_{i\in\{N_l/2+1,\cdots,N_l\}}$, respectively, with each employing only $N_l/2$ particles to approximate the mean-field term.
We let $N_{l+1}=2N_l$ and the number of paths $U=500$.
The numerical simulations are conducted with terminal times 
$T=1$ and $T=30$ to verify Theorem \ref{keythm} and Theorem \ref{longpoc}, respectively.
The results are depicted in Figure 1.
\begin{figure}[htbp]
	\centering
	\subfigure[$T=1$]{
		\includegraphics[width=0.45\linewidth]{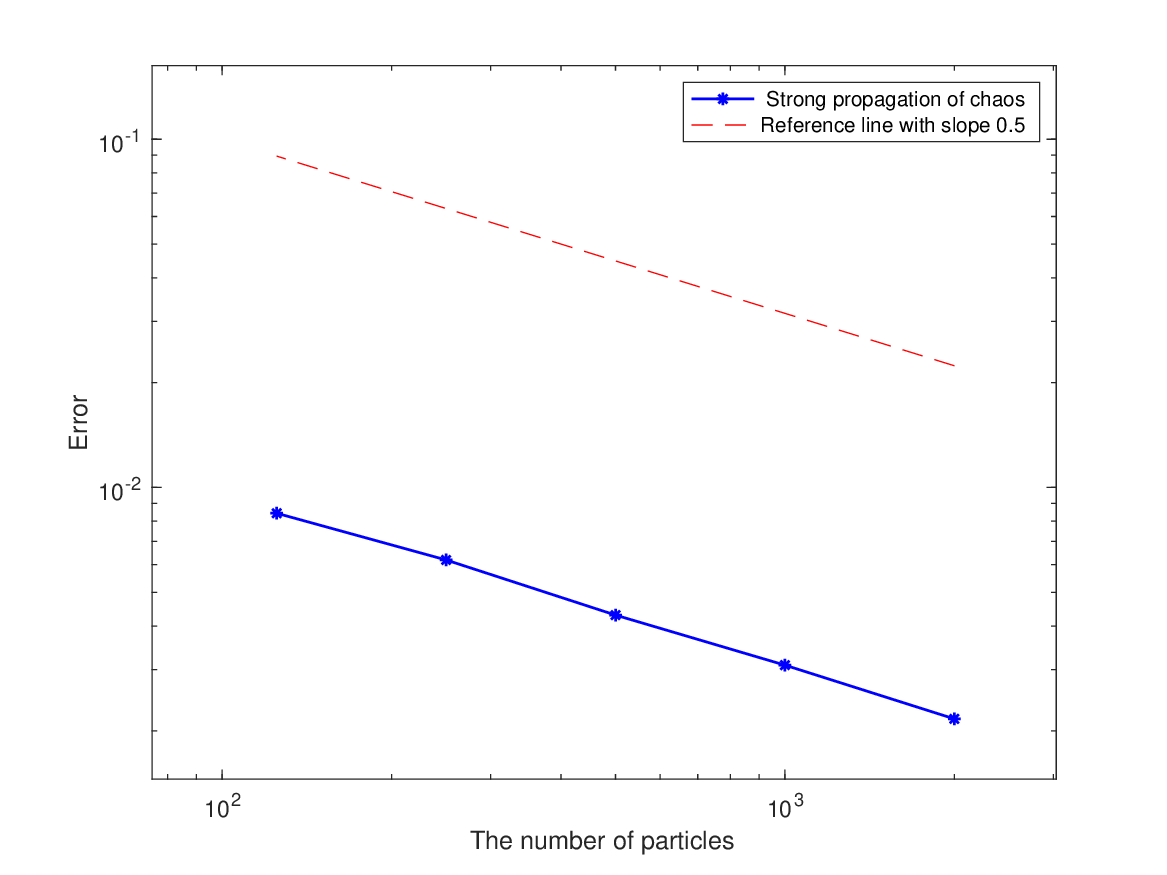}
		\label{tu0300}
	}
	\hfill
	\subfigure[$T=30$]{
		\includegraphics[width=0.45\linewidth]{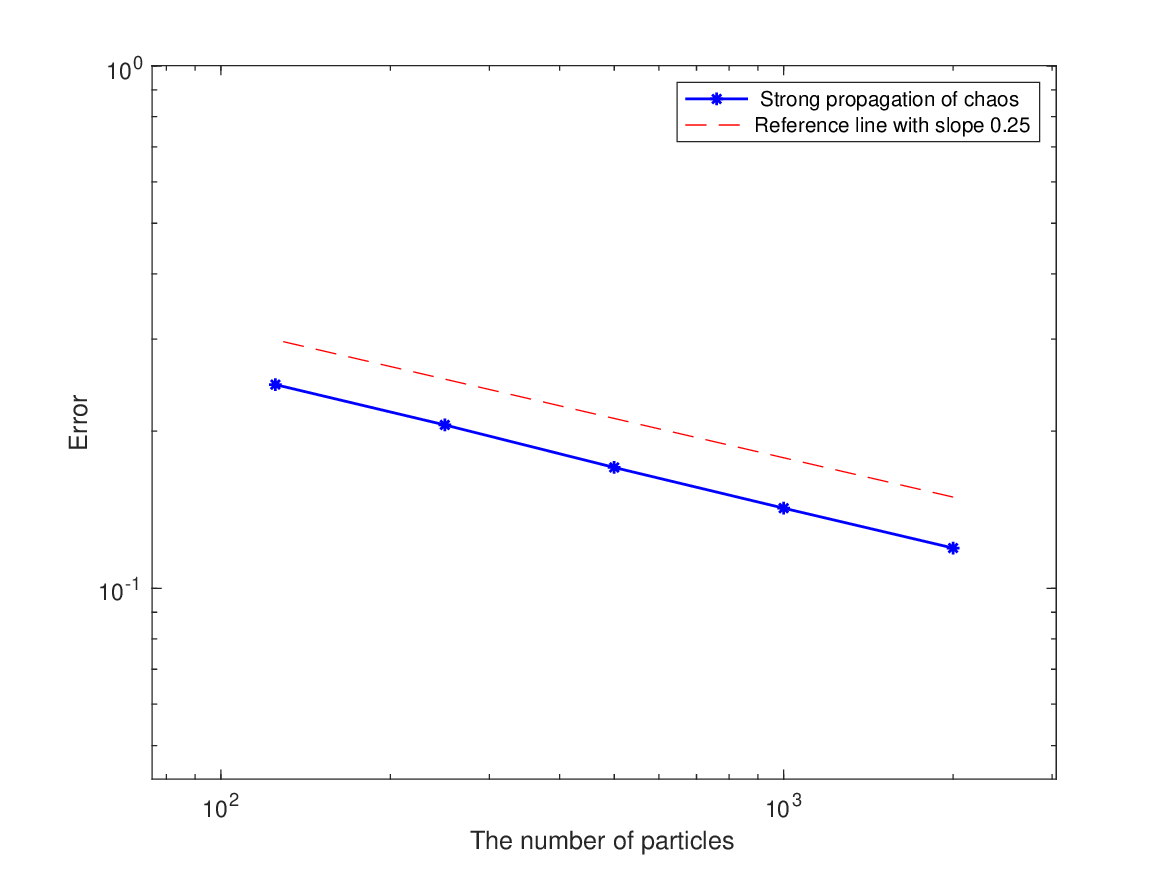}
		\label{tu0400}
	}
	\caption{Strong convergence with respect to the number of particles }
	\label{fig:scheme_comparison}
\end{figure}

\section{Appendix}\label{sec}		
\subsection{Proof of the exponential integrability stated in Remark \ref{rree01} }\label{aa6.1}
Let Assumption \ref{ass6} hold.
Define $V(t,X_t)=exp\left(e^{-\alpha t}f(|X_t|)\right)$. By Ito's formula, we have
\begin{equation}\label{exp}
	\begin{split}
		&V(t,X_t)-V(0,\xi)\\
		=&\int_{0}^{t} e^{-\alpha  s}  V(s,X_s)\Big[\langle \nabla f(|X_s|),b(X_s,\mu_s)
		\rangle+\frac{e^{-\alpha s}}{2}|\nabla f (|X_s|)\sigma(X_s,\mu_s) |^2\\&+\frac{1}{2}\text{trace}\left({\sigma}^{T}(X_s,\mu_s)\nabla^2 f(|X_s|)\sigma(X_s,\mu_s)\right)-\alpha f(|X_s|)\Big] ds\\
		&+\int_{0}^{t} e^{-\alpha  s}  V(s,X_s) \langle \nabla f (|X_s|),\sigma(X_s,\mu_s)dW_s\rangle.
	\end{split}
\end{equation}
After taking expectation on both sides of (\ref{exp}) and using Assumption  \ref{ass6}, we arrive at
\begin{equation*}
	\begin{split}
		&\mathbb{E}[V(t,X_t)]-\mathbb{E}[V(0,\xi)]\\\leq&\mathbb{E}\left[\int_{0}^{t} e^{-\alpha  s}  V(s,X_s) \langle \nabla f (|X_s|),\sigma(X_s,\mu_s)dW_s\rangle\right]\\&+\beta\mathbb{E}\left[\int_{0}^{t} e^{-\alpha  s}  V(s,X_s)\bar{f}(\|\mu_s\|_2^2)ds\right].
		%\\\leq&\mathbb{E}\left[\int_{0}^{t} e^{-\alpha  s}  V(t,X_s) \langle(\nabla U)(X_s),\sigma(X_s,\mu_s)dW_s\rangle\right]+\beta\int_{0}^{t}f(\|\mu_s\|_2^2)ds+\beta T.
	\end{split}
\end{equation*}
Since $\mu_s\in\mathcal{P}_2(\mathbb{R}^d)$, $\|\mu_s\|_2^2$ is a continuous function of $s$ on $[0,t]$, and $\bar{f}$ is a polynomial function, it follows that $\int_{0}^{t}\bar{f}(\|\mu_s\|_2^2)ds<\infty$. Then, it follows that
%employing stopping time theorem leads to 
\begin{equation*}\label{U}
	\begin{split}
		\sup_{0\leq t\leq T}	\mathbb{E}[V(t,X_t)]&\leq\mathbb{E}[V(0,\xi)]+\beta\mathbb{E}\left[\int_{0}^{T} e^{-\alpha  t}  V(t,X_t)\bar{f}(\|\mu_t\|_2^2)dt\right].
		%+\beta\int_{0}^{t}f(\mathbb{E}|X_s|^2)ds+\beta T\\&\leq\mathbb{E}[e^{U(\xi)}]+\beta\int_{0}^{t}\mathbb{E}\left[f(|X_s|^2)\right]ds+\beta T,
	\end{split}
\end{equation*}
Consequently, by Gronwall's inequality, we obtain
%the feature of function $\bar{f}$ gives
$$	\sup_{0\leq t\leq T}	\mathbb{E}[V(t,X_t)]\leq\mathbb{E}[V(0,\xi)]\exp\left(\beta\int_{0}^{t} e^{-\alpha  s}  \bar{f}(\|\mu_s\|_2^2)ds\right)<\infty,$$
which means
$
\sup_{0\leq t\leq T}\mathbb{E}[\exp({f(|X_t|)})]<\infty,
$
so the exponential integrability is shown.

\subsection{Verification  of coefficient conditions for Example (\ref{exm}) in Remark \ref{numrem}}
Let $$b(x,\mu)=-18x^5-x^\frac{1}{3}\left(\int_{\mathbb{R}}x\mu(dx)\right)^4+2$$ and
$$\sigma(x,\mu)=x^2+\int_{\mathbb{R}}x\mu(dx).$$
Since Assumptions \ref{ass4} and \ref{ass2} can be easily verified, we only provide detailed verification for Assumptions \ref{ass1}, \ref{ass3}, and \ref{ass6}.
\subsubsection{Verification of  Assumption \ref{ass1}}
One can verify that
\begin{equation*}
	\begin{split}
		&2\langle x-y, b(x,\mu)-b(y,\nu)\rangle+|\sigma(x,\mu)-\sigma(y,\nu)|^2\\
		\leq&-36\langle x-y, x^5-y^5\rangle-2\langle x-y, x^\frac{1}{3}\big(\int_{\mathbb{R}}x\mu(dx)\big)^4-y^\frac{1}{3}\big(\int_{\mathbb{R}}y\nu(dy)\big)^4\rangle\\&+2|x^2-y^2|^2+2\big|\int_{\mathbb{R}}x\mu(dx)-\int_{\mathbb{R}}y\nu(dy)\big|^2\\
		\leq&-36(x^4+x^3y+x^2y^2+xy^3+y^4)| x-y|^2+2|x+y|^2|x-y|^2\\&-2\langle x-y, x^\frac{1}{3}\big(\int_{\mathbb{R}}x\mu(dx)\big)^4-y^\frac{1}{3}\big(\int_{\mathbb{R}}x\mu(dx)\big)^4\rangle\\&-2\langle x-y, y^\frac{1}{3}\big(\int_{\mathbb{R}}x\mu(dx)\big)^4-y^\frac{1}{3}\big(\int_{\mathbb{R}}y\nu(dy)\big)^4\rangle+2\mathbb{W}_{1}^{2}(\mu,\nu)\\
		\leq&(-36x^4-36x^3y-36x^2y^2-36xy^3-36y^4+4x^2+4y^2)| x-y|^2+2\mathbb{W}_{1}^{2}(\mu,\nu)\\&+4|x-y| |y|^\frac{1}{3}\big|\int_{\mathbb{R}}x\mu(dx)-\int_{\mathbb{R}}y\nu(dy)\big|\big(\big|\int_{\mathbb{R}}x\mu(dx)\big|^3+\big|\int_{\mathbb{R}}y\nu(dy)\big|^3\big)\\
		%	\leq&6\big(\int_{\mathbb{R}}x\mu(dx)\big)^4| x-y|^2+12|x-y| |y|\big[\big(\int_{\mathbb{R}}x\mu(dx)\big)^3+\big(\int_{\mathbb{R}}y\nu(dy)\big)^3\big]\big[\int_{\mathbb{R}}x\mu(dx)-\int_{\mathbb{R}}y\nu(dy)\big]
		%			\\&+2|\int_{\mathbb{R}}x\mu(dx)-\int_{\mathbb{R}}y\nu(dy)|^2\\
		\leq&(-36x^4-36x^3y-36x^2y^2-36xy^3-36y^4+4x^2+4y^2)| x-y|^2+2\mathbb{W}_{1}^{2}(\mu,\nu)\\&+\big(|y|^\frac{2}{3}+2\|\mu\|_6^6+2\|\nu\|_6^6\big)\big(|x-y|^2 +\mathbb{W}_{1}^{2}(\mu,\nu)\big).
		%	\leq&\big(1+|y|^\frac{1}{3}+2\|\mu\|_6^6+2\|\nu\|_6^6\big)\big(|x-y|^2 +\mathbb{W}_{2}^{2}(\mu,\nu)\big)
	\end{split}
\end{equation*}
Define 
$$\varphi(x,y)=-36x^4-36x^3y-36x^2y^2-36xy^3-36y^4+4x^2+4y^2.$$
Through the analysis of function $\varphi(x,y)$, one can show that
\begin{equation*}
	\begin{split}
		&\varphi(x,y)\leq \frac{4}{9}.
	\end{split}
\end{equation*}
Thus,
\begin{equation*}
	\begin{split}
		&2\langle x-y, b(x,\mu)-b(y,\nu)\rangle+|\sigma(x,\mu)-\sigma(y,\nu)|^2\\
		\leq&\big(2+|y|^\frac{2}{3}+2\|\mu\|_6^6+2\|\nu\|_6^6\big)\big(|x-y|^2 +\mathbb{W}_{2}^{2}(\mu,\nu)\big),
	\end{split}
\end{equation*}
which implies that the drift and diffusion satisfy the Assumption \ref{ass1} with $L(R)=R^\frac{2}{3}+2$.
% and the definition of $\mathbb{W}_1$ Wasserstein distance in \cite{W1} is used here. Additionally, 
\subsubsection{Verification  of Assumption \ref{ass3}}
It is straightforward to check that
\begin{equation*}
	\begin{split}
		&2\langle x, b(x,\mu)\rangle+|\sigma(x,\mu)|^2\\
		\leq&-36|x|^6-2x^\frac{4}{3}\big(\int_{\mathbb{R}}x\mu(dx)\big)^4+4x+2|x|^4+2\big(\int_{\mathbb{R}}x\mu(dx)\big)^2\\
		\leq&6+2\|\mu\|_2^2,
	\end{split}
\end{equation*}			
which means that Assumption \ref{ass3} holds. 
\subsubsection{Verification  of Assumption \ref{ass6}}
Let $f(|x|)=\frac{1}{4}|x|^2+1$.		
It is obvious that, for $R>0$, 
$$
\lim_{R\rightarrow\infty}\big(f(R)-L(R)\big)=\infty,
$$
so (\ref{R}) holds with $\kappa=1$. Additionally, we have
\begin{equation*} 
	\begin{split}
		&	\langle \nabla f(|x|),b(x,\mu)
		\rangle+\frac{1}{2}|\sigma(x,\mu)\nabla f(|x|)|^2+\frac{1}{2}\text{trace}\left({\sigma}^{T}(x,\mu)\nabla^2 f(|x|)\sigma(x,\mu)\right)\\
		\leq&-9|x|^6-\frac{1}{2}|x|^\frac{4}{3}\big(\int_{\mathbb{R}}x\mu(dx)\big)^4\\&+x+\frac{1}{4}|x|^6+\frac{1}{4}|x|^2\big(\int_{\mathbb{R}}x\mu(dx)\big)^2+\frac{1}{2}|x|^4+\frac{1}{2}\big(\int_{\mathbb{R}}x\mu(dx)\big)^2\\
		\leq&\frac{1}{2}|x|^2+\frac{3}{2}+\|\mu\|_2^3+\|\mu\|_2^2
		\\
		=&
		2f(|x|)+\bar{f}(\|\mu\|_2^2),
	\end{split}
\end{equation*} 
with $\alpha=2$, $\beta=1$ and $\bar{f}(x)=(x)^{\frac{3}{2}}+x^2$.

\section*{Funding}
This work is supported by the National Natural Science Foundation of
China (12271368, 12501579 and 62373383), the Fundamental Research Funds
for the Central Universities of South-Central MinZu University (CZQ25020, YZY24010, CZZ25007), and Fund for Academic Innovation Teams of South-Central
Minzu University (XTZ24004).
%\bibliographystyle{plain}
%\bibliography{reference}

\section*{Declaration of competing interest}
The authors declare that they have no known competing financial interests or personal relationships that could have appeared to influence the work reported in this paper.

\section*{Data availability}
No data was used for the research described in the paper.

\end{document}